\documentclass[reqno,a4paper,10pt]{amsart}
\usepackage[utf8]{inputenc} 
\usepackage{amsmath, amssymb,bm, cases, mathtools, thmtools}
\usepackage{verbatim}
\usepackage{graphicx}\graphicspath{{figures/}}
\usepackage{multicol}
\usepackage{tabularx}
\usepackage[usenames,dvipsnames]{xcolor}
\usepackage{mathrsfs}

\usepackage[%
    minnames=1,maxnames=99,maxcitenames=3,
    style=alphabetic,
    doi=false,url=false,
    firstinits=true,hyperref,natbib,backend=bibtex]{biblatex}
\renewbibmacro{in:}{%
  \ifentrytype{article}{}{\printtext{\bibstring{in}\intitlepunct}}}
\bibliography{continuumofurns}

\usepackage[colorlinks,citecolor=blue,urlcolor=blue,linkcolor=RawSienna,hypertexnames=false]{hyperref}
\usepackage{hypernat}
\usepackage{datetime}
\usepackage{enumerate}

\DeclareMathAlphabet\EuRoman{U}{eur}{m}{n}
\SetMathAlphabet\EuRoman{bold}{U}{eur}{b}{n}
\newcommand{\eurom}{\EuRoman}

\usepackage{euscript,microtype}

\makeatletter
\let\reftagform@=\tagform@
\def\tagform@#1{\maketag@@@{\ignorespaces\textcolor{gray}{(#1)}\unskip\@@italiccorr}}
\renewcommand{\eqref}[1]{\textup{\reftagform@{\ref{#1}}}}
\makeatother

\newcommand{\LATER}[1]{\error}
\newcommand{\fLATER}[1]{\error}
\newcommand{\TBD}[1]{\error}
\newcommand{\fTBD}[1]{}
\newcommand{\PROBLEM}[1]{\error}
\newcommand{\fPROBLEM}[1]{}
\newcommand{\NA}[1]{#1}

\newcommand{\charbersym}[1]{\Psi_{#1}}
\newcommand{\charber}[2]{\charbersym {} \ooF{#1,\ #2}}

\def\[#1\]{\begin{align}#1\end{align}}
\newcommand{\defas}{\vcentcolon=}  
\newcommand{\given}{\mid}

\newcommand{\Reals}{\mathbb{R}}
\newcommand{\BorelSets}{\mathcal{B}}

\newcommand{\as}{\textrm{a.s.}}
\newcommand{\AS}{\ \,\as}

\newcommand{\downto}{\!\downarrow\!}
\newcommand{\upto}{\!\uparrow\!}
\newcommand{\ind}{\mathrel{\perp\mkern-9mu\perp}}

\newcommand{\dee}{\mathrm{d}}

\DeclareMathOperator{\supp}{supp}

\DeclareMathOperator*{\newlim}{\mathrm{lim}\vphantom{\mathrm{infsup}}}

\DeclareMathOperator*{\newinf}{\mathrm{inf}\vphantom{\mathrm{infsup}}}
\DeclareMathOperator*{\newsup}{\mathrm{sup}\vphantom{\mathrm{infsup}}}
\renewcommand{\lim}{\newlim}

\renewcommand{\inf}{\newinf}
\renewcommand{\sup}{\newsup}

\newcommand{\Measures}{\mathcal{M}}

\newcommand{\Nats}{\mathbb{N}}

\newcommand{\Ints}{\mathbb{Z}}
\newcommand{\NNInts}{\Ints_+}
\newcommand{\NNExtInts}{\overline{\Ints}_+}

\newcommand{\st}{\,:\,}
\renewcommand{\Pr}{\mathbb{P}}
\newcommand{\defn}[1]{\emph{#1}}

\newcommand{\ceiling}[1]{\lceil #1 \rceil}

\def\bone{\mathbf{1}}
\def\Ind{\bone}

\def\EE{\mathbb{E}}

\def\law{\mathcal{L}}

\newcommand{\cF}{\mathcal F}
\newcommand{\cG}{\mathcal G}

\newcommand{\Uniform}{\textrm U (0,1)}

\newcommand{\equaldist}{\overset{d}{=}}
\newcommand{\inv}{^{-1}}

\newcommand{\dist}{\ \sim\ }

\newcommand{\uspace}{\Xi}
\newcommand{\upoint}{x}
\newcommand{\usa}{\mathcal{F}}

\newcommand{\probspace}{(\uspace,\usa,\Pr)}

\newcommand{\bspace}{\Omega}
\newcommand{\bsa}{\mathcal A}
\newcommand{\borelspace}{(\bspace,\bsa)}
\newcommand{\bp}{s}
\newcommand{\bset}{A}
\newcommand{\bseta}{A'}

\newcommand{\pp}[1]{\tilde #1}

\newcommand{\NNReals}{\Reals_+}
\newcommand{\thelaw}{\!_{\!\mathcal{L}}}

\newcommand{\Beta}{\mathrm{Beta}}
\newcommand{\Bernoulli}{\mathrm{Bernoulli}}
\newcommand{\BPLAW}{\mathrm{BP}_{\thelaw}}
\newcommand{\BePLAW}{\mathrm{BeP_{\thelaw}}}

\newcommand{\distiid}{\overset{iid}{\dist}}

\newcommand{\card}[1]{\##1}
\newcommand{\iid}{i.i.d.}
\newcommand{\CombStruct}[1]{[\![#1]\!]}

\newcommand{\gprocess}[2]{(#1)_{#2}}
\newcommand{\nprocess}[3]{\gprocess{#1_{#3}}{#3 \in #2}}
\newcommand{\process}[2]{\nprocess{#1}{#2}n}

\newcommand{\Atoms}{\mathscr{A}}
\newcommand{\FixedAtoms}{\mathscr{A}_0}

\newcommand{\theset}[1]{\lbrace #1 \rbrace}

\newcommand{\BM}{H_0} 
\newcommand{\NABM}{\ord H_0}

\newcommand{\nUni}{K}
\newcommand{\nmul}{N}

\newcommand{\kernel}{\Bbbk}
\newcommand{\bkernel}{\overline{\kernel}}

\newcommand{\eppf}{\pi}
\newcommand{\EPPFfam}{\nprocess \pi \bspace s}

\newcommand{\of}[1]{\left ( #1 \right )}
\newcommand{\oof}[1]{( #1 )}
\newcommand{\ooof}[1]{\bigl ( #1 \bigr )}
\newcommand{\oooof}[1]{\Bigl ( #1 \Bigr )}

\newcommand{\oF}[1]{\left [ #1 \right ]}
\newcommand{\ooF}[1]{[ #1 ]}
\newcommand{\oooF}[1]{\bigl [ #1 \bigr ]}
\newcommand{\ooooF}[1]{\Bigl [ #1 \Bigr ]}
\newcommand{\oooooF}[1]{\biggl [ #1 \biggr ]}

\newcommand{\Of}[1]{\left \{ #1 \right \}}
\newcommand{\Oof}[1]{\{ #1 \}}
\newcommand{\Ooof}[1]{\bigl \{ #1 \bigr \}}

\newcommand{\Ooooof}[1]{\biggl \{ #1 \biggr \}}

\newcommand{\nalpha}[1]{\eurom{#1}}
\newcommand{\nX}{\nalpha X}
\newcommand{\nY}{\nalpha Y}
\newcommand{\nR}{\nalpha R}

\newcommand{\cQkernel}{\mathbb Q}
\newcommand{\cQ}[1]{\cQkernel_{#1}} 

\newcommand{\ginitial}[2]{\tau^{#2}_{#1}}
\newcommand{\initial}[1]{\ginitial {#1} {}}

\newcommand{\ftt}[2]{\ft^{#1}_{#2}}

\usepackage[capitalize]{cleveref}

\crefname{lem}{Lemma}{Lemmas}
\crefname{cor}{Corollary}{Corollaries}
\crefname{thm}{Theorem}{Theorems}
\crefname{assumption}{Assumption}{Assumptions}

\numberwithin{equation}{section}

\declaretheorem[style=plain,numberwithin=section,name=Theorem]{thm}
\declaretheorem[style=plain,sibling=thm,name=Lemma]{lem}
\declaretheorem[style=plain,sibling=thm,name=Proposition]{prop}
\declaretheorem[style=plain,sibling=thm,name=Corollary]{cor}

\declaretheorem[style=definition,sibling=thm,name=Definition]{definition}

\declaretheorem[style=remark,qed=$\triangleleft$,sibling=thm,name=Remark]{remark}

\numberwithin{thm}{section}

\usepackage{url}

\newcommand{\indexmk}[4]{#1_{#2,#3}^{#4}}
\newcommand{\indexm}[3]{#1_{#2,#3}}
\newcommand{\indexk}[3]{#1_{#2}^{#3}}
\newcommand{\ord}[1]{\tilde{#1}}
\newcommand{\countproc}[1]{\mathcal C(#1)}

\newcommand{\BorelSetsInt}{\BorelSets_{[0,1]}}

\newcommand{\PFF}[1]{\rho^{(#1)}}
\newcommand{\PF}[2]{\rho^{(#1,#2)}}
\newcommand{\Bin}{\mathrm{Bin}}

\def\pp{p}
\def\qq{q}

\newcommand{\Triangular}{Tetrahedral}
\newcommand{\triangular}{tetrahedral}

\newcommand{\gsd}[1]{\varsigma_{#1}}
\newcommand{\sdist}{\gsd{1}}

\newcommand{\nrm}{\mu}
\newcommand{\lbm}{\upsilon}
\newcommand{\lkernel}{\kappa}

\newcommand{\mdrm}{\nu}
\newcommand{\drm}{\tilde \mdrm}

\newcommand{\rcard}{\zeta}

\newcommand{\thekernel}{\Bbbk}

\newcommand{\Gorur}{G\"or\"ur}
\newcommand{\TG}{Teh and G\"or\"ur}

\newcommand{\conc}{\theta}
\newcommand{\disc}{\alpha}

\newcommand{\rH}{H}

\newcommand{\hazard}{hazard}
\newcommand{\bpmean}{mean}
\newcommand{\bermean}{\hazard\ measure}
\newcommand{\cupmean}{mean \bermean}
\newcommand{\nonrandomcomp}{nonrandom nonatomic}
\newcommand{\fixedcomp}{fixed-atomic}

\newcommand{\argdot}{\,\cdot\,}

\newcommand{\indicator}[2]{e_{#1,#2}}

\newcommand{\generalizedbetaprocess}{generalized beta process}

\newcommand{\drag}[1]{\hat {#1}}
\newcommand{\normalize}[1]{\overline{#1}}

\newcommand{\general}{lcscH}%
\newcommand{\COUS}{continuum-of-urns scheme}

\title
[The continuum-of-urns scheme]
{
The continuum-of-urns scheme, \\
generalized beta and Indian buffet processes, \\
and hierarchies thereof
}

\begin{document}

\author[D.~M.~Roy]{Daniel M.~Roy}
\address{University of Toronto}
\urladdr{http://danroy.org/}
\email{droy@utstat.toronto.edu}

\begin{abstract}
We describe the combinatorial stochastic process underlying a sequence of conditionally independent Bernoulli processes with a shared beta process \hazard\ measure.  As shown by Thibaux and Jordan~\citep{Thibaux2007}, in the special case when the underlying beta process has a constant concentration function and a finite and nonatomic mean, the combinatorial structure is that of the Indian buffet process (IBP) introduced by Griffiths and Ghahramani \citep{GG05}.  
By reinterpreting the beta process introduced by Hjort~\citep{MR1062708} as a measurable family of Dirichlet processes, we obtain a simple predictive rule for the general case, which can be thought of as a continuum of Blackwell--MacQueen urn schemes (or equivalently, one-parameter Hoppe urn schemes). The corresponding measurable family of Pitman--Yor processes leads to a continuum of two-parameter Hoppe urn schemes, whose ordinary component is the three-parameter IBP introduced by Teh and G\"or\"ur \citep{TehGor2009a}, which exhibits power-law behavior, as further studied by Broderick, Jordan, and Pitman~\citep{BJP12}.  The idea extends to arbitrary measurable families of exchangeable partition probability functions and gives rise to generalizations of the beta process with matching buffet processes.  Finally, in the same way that hierarchies of Dirichlet processes were given Chinese restaurant franchise representations by Teh, Jordan, Beal, and Blei~\citep{TehJorBea2006}, one can construct representations of sequences of Bernoulli processes directed by hierarchies of beta processes (and their generalizations) using the stochastic process we uncover.
\end{abstract}

\maketitle

\thispagestyle{empty}

\ \\

\begin{center}
\begin{minipage}{.80\linewidth}
\setcounter{tocdepth}{1}
\tableofcontents
\end{minipage}
\end{center}

\vfill

\begin{center}
{\footnotesize This draft is subject to change.}
\end{center}

\newpage

\section{Introduction}
\label{sec:intro}

Since the introduction of the Indian buffet process (IBP) by Griffiths and Ghahramani \citep{GG05,GG06} and the characterization of its relationship with beta and Bernoulli processes by Thibaux and Jordan \citep{Thibaux2007}, there has been a surge of work 
extending the IBP in one direction and further exploiting the theory of completely random measures in the other.  
Despite this attention, a characterization of an urn scheme corresponding to a hierarchy of beta processes has remained elusive, in part, because of the family of beta distributions is not self-conjugate.  By reinterpreting the beta process as a measurable family of Dirichlet processes, we obtain such an urn scheme, which we subsequently generalize by considering arbitrary random measures.
As the main example, the urn scheme arising from Pitman--Yor processes not only gives rise to the 
stable beta process and three-parameter IBP introduced by Teh and G\"or\"ur \citep{TehGor2009a}, but also 
gives rise to a canonical definition for a hierarchical stable beta process.

In this article, we study exchangeable sequences of random sets, their combinatorial structure, and their corresponding de~Finetti (mixing) measures.
Following \citep{BJP13}, we will refer to the combinatorial structure of a collection of finite sets as a \emph{feature allocation}.  Informally, a feature allocation is the Venn diagram adorned with counts for each component. 

It will be convenient to represent random subsets of a space $\bspace$ by random measures on $\bspace$.  
In particular, a so-called \emph{simple} point process $X$ of the form $X=\sum_{k \le \rcard} \delta_{\gamma_k}$, 
for some random element
$\rcard$ in $\NNExtInts \defas \NNInts \cup \{\infty\}$
and 
\as\ distinct random elements
$\gamma_1,\gamma_2,\dotsc$ in $\bspace$,
will be taken to represent the set $\theset{\gamma_k \st k \le \rcard }$ of its atoms.
We will assume that $\bspace$ is locally compact, second countable, and Hausdorff (abbreviated \emph{lcscH}).  Let $\bsa$ denote the $\sigma$-algebra of its Borel sets.
The corresponding $\sigma$-algebra on the space of measures on $\borelspace$ is that generated by the evaluation maps
$\pi_\bset : \mu \mapsto \mu \bset$, for $\bset \in \bsa$.  Alternatively, we may think of random subsets as random elements in the space of $\sigma$-finite subsets of $\bspace$, equipped with the $\sigma$-algebra generated by the maps $\bset \mapsto \card{ (\bseta \cap \bset )}$, for $\bseta \in \bsa$, where $\card \bset$ denotes the cardinality of the set $\bset$.
Recall that a random measure $\mu$ is said to be \defn{completely random} or, equivalently, have \defn{independent increments}, 
when the random variables $\mu(A_j)$ are independent for any disjoint collection $A_1,\dotsc,A_k$ of measurable subsets.\footnote{
We have adopted the framework for random measures, point processes, Poisson processes, and more generally, random measures with independent increments---also known as completely random measures---laid out by Kallenberg~\citep[][Chp.~12]{FMP2}.  
This background, as well as some results on beta and Bernoulli processes, two classes of completely random measures of particular interest in the first part of the article, are presented in \cref{sec:background}.}
By a \defn{hazard measure}, we will mean a $\sigma$-finite measure $\mu$ on $\borelspace$ such that $\mu \theset s \le 1$ for all $s \in \bspace$.

\subsection{A discrete model}
\label{sec:finmodel}

We begin with a simple model.  Fix a finite, purely-atomic \hazard\ measure $\BM$, let $\Atoms$ be the set of its atoms,
let $\Pi$ be a random partition of $\Nats \defas \theset {1,2,\dotsc}$, 
and let $\Pi^s$, for $s \in \Atoms$, be independent and identically-distributed (\iid)\ copies of $\Pi$.  
The partition $\Pi^s$ associated with each atom $s \in \Atoms$ is  
a random finite or countably-infinite collection $C^s_1,C^s_2,\dotsc$ of disjoint subsets of $\Nats$, called \emph{blocks}.
Let $U^s_n$, for $s \in \Atoms$ and $n \in \Nats$, be an \iid\ collection of uniformly-distributed random variables, independent also from the partitions.
Then consider the sequence of simple point processes $X_n$, for $n \in \Nats$, concentrated on $\Atoms$ and given by 
\[
X_n \theset s \defas 1(U^s_{k^s_n} \le \BM\theset s),\ \text{where $k^s_n$ satisfies $n \in C^s_{k^s_n}$.}
\]  
Informally, every block in the partition $\Pi^s$ ``decides'' independently with probability $\BM\theset s$ whether or not to ``take'' the atom $s \in \Atoms$.  Then the set $X_n$ has the atom $s$ if and only if the block containing $n$ in $\Pi^s$ has itself ``taken'' the atom.
As constructed, $\EE X_n = \BM$ and $X_n \theset s = X_m \theset s$ whenever $n,m$ are in the same block of $\Pi^s$.

We are interested in the law of $\process X \Nats$ under the additional assumption that the random partition $\Pi$ is \emph{exchangeable} in the sense that its distribution is invariant under every permutation of the underlying set $\Nats$.  More carefully, by a \defn{random partition of} $\Nats$ we mean a $\{0,1\}$-valued process $\Pi \defas \nprocess \Pi \Nats {i,j}$ such that the random set $\theset { (i,j) \in \Nats^2 \st \Pi_{i,j}= 1}$ is an equivalence relation on $\Nats$ with probability one.
We say that a random partition $\Pi$ is \defn{exchangeable} when, for all permutations $\sigma$ of $\Nats$,
\[
\nprocess \Pi \Nats {i,j} 
   \equaldist 
\gprocess {\Pi_{\sigma(i),\sigma(j)}} {i,j \in \Nats}.
\]
If $\Pi$ is an exchangeable, one can show that the sequence $\process X \Nats$ is itself exchangeable, and thus conditionally i.i.d.  In particular, there exists a completely-random, purely-atomic, \hazard\ measure $\rH$, concentrated on $\Atoms$ and given by 
\[\label{limitingdefn}
\rH(A) = \lim_{n \to \infty} \frac 1 n \sum_{j=1}^n X_n(A)\ \as, \quad A \in \bsa,
\]
such that, conditioned on $\rH$, the $X_n$ are completely random with mean $\rH$.  It follows that $\rH$ is the \as\ unique random measure with this property.  (We will say that $\rH$ \emph{directs} $\process X \Nats$.)  

The distribution of $\process X \Nats$ and the directing random hazard measure $\rH$ are measurable functions of $\BM$ and the distribution of the random partition $\Pi$.  As one example, if $\Pi$ is a random partition induced by a \emph{one-parameter Chinese restaurant process}~(CRP)~\citep{MR2245368}, then $\rH$ is the \fixedcomp\ component of a \emph{beta process} \citep{MR1062708,MR1714717,Thibaux2007} with mean $\BM$.
Write $\cQkernel_{\BM}$ for the distribution of $\process X \Nats$, where we have highlighted only its dependence on $\BM$.

\subsection{The continuum limit}

The main focus of this article is on the characterization of 
a sequence of random measures whose distribution can be obtained by the following limit construction:  
Let $\BM^1,\BM^2,\dotsc$ be a sequence of purely-atomic \hazard\ measures on $\borelspace$ converging \emph{strongly} to a $\sigma$-finite, though not necessarily purely atomic, measure $\BM$,  
and write $\cQkernel_{\BM}$ for the weak limit of the distributions $\cQkernel_{\BM^k}$ as $k \to \infty$.
We will call a sequence $\process X \Nats$ of random measures with distribution $\cQkernel_{\BM}$
a (homogeneous) \defn{\COUS} with \hazard\ measure $\BM$. 

For the remainder of the section, we present our main results characterizing the continuum limit.
In \cref{sec:oneparam}, we will give a direct construction of the \COUS\ in the special case where the random partition is that induced by a Chinese restaurant process.
In \cref{sec:cup}, we give a direct construction of a general nonhomogeneous \COUS\ without appealing to a limiting argument, where nonhomogeneity refers to the fact that the distribution of the random partition $\Pi$ is allowed to vary across $\bspace$.
In \cref{sec:climit}, we show that the weak continuum limit, outlined above, agrees with these constructions.

\begin{thm}\label{intromain}
Let $\process X \Nats$ be a \COUS\ with \hazard\ measure $\BM$.
Then $\process X \Nats$ is exchangeable, and thus conditionally i.i.d.  In particular, there exists an \as\ unique, random \hazard\ measure $\rH$, given by \cref{limitingdefn}, such that, conditioned on $\rH$, the $X_n$, for $n \in \Nats$, are \iid\ and completely random with mean $H$.
\end{thm}

Let $\process X \Nats$ and $\rH$ be given as above.  We will say that $\rH$ \emph{directs} $\process X \Nats$ and will call such a random \hazard\ measure a (homogeneous) \defn{\generalizedbetaprocess}.  (Nonhomogeneous \generalizedbetaprocess{}es will arise as the random \hazard\ measures directing nonhomogeneous \COUS{s}.)  
Before we can characterize the law of such processes, we must introduce a few notions from the theory of exchangeable sequences.  (We will develop these concepts further in \cref{sec:esp}.)  

Let $\Pi = \theset{C_1,C_2,\dotsc}$ be an exchangeable partition of $\Nats$,
where $C_1$ is the block containing $1$ and $C_{k+1}$, for $k\in\Nats$, is the block that, when nonempty, contains the least integer not in $C_1 \cup \dotsm \cup C_k$. 
Let $[n] \defas \{1,\dotsc,n\}$ and let $N_{jn} \defas \card \of { [n] \cap C_j }$, for $j \in \Nats$, be the number of elements in block $C_j$ among $[n]$.
Then the limiting relative frequency of elements in block $C_k$, i.e., 
\[
P_k \defas \lim_{n \to \infty} \frac{N_{kn}}{n} \ \as,
\]
exists almost surely for every $k \in \Nats$.  
Let $\sdist$ be the \emph{structural distribution}, i.e., the distribution of the first size-biased pick $P_1$, 
let
\[
\Delta \defas \EE(1-\textstyle\sum_{n=1}^\infty P_n)
\]
be the expected limiting frequency of \emph{dust}, i.e., singleton blocks in $\Pi$, 
and let $\thekernel(\qq, \argdot)$ be the distribution of
\[
\sum_{n=1}^\infty P_n 1(U_n \le \qq) 
+ \oooof { 1 -  \sum_{n=1}^\infty P_n }\,\qq
,
\]
where $\process U \Nats$ is an independent \iid\ process of uniformly-distributed random variables (cf., the discrete model). %
We then have the following:

\begin{thm}\label{genbeta}
Let $\rH$ be the random \hazard\ measure directing a \COUS\ 
with \hazard\ measure $\BM$, and let $\FixedAtoms$ and $\NABM$ be the atoms and nonatomic part of $\BM$, respectively.
Then $\rH$ is completely random and can be written
\[\label{decomp1}
\rH = \Delta \NABM + \sum_{s \in \Atoms} \tilde \pp_s \delta_s + \sum_{(s,\pp) \in \eta} \!\! \pp\, \delta_s \AS
\]
where $\nprocess {\tilde \pp} {\Atoms} s$ is a process of independent random variables such that $\tilde \pp_s$ has distribution $\thekernel(\BM\theset s, \argdot)$, and $\eta$ is a Poisson process on $\bspace \times (0,1]$, independent from $\nprocess {\tilde \pp} {\Atoms} s$, with intensity measure
\[\label{levy}
(\dee s, \dee \pp) \mapsto \NABM(\dee s) \, \pp\inv \, \sdist(\dee \pp).
\]
\end{thm}

Following convention in the study of completely random measures, we will call the three components of $\rH$ appearing in \cref{decomp1} the \emph{\nonrandomcomp}, \emph{\fixedcomp}, and \emph{ordinary} components, respectively.  
When $\Delta < 1$ and $\NABM \neq 0$, the measure described by \cref{levy} is merely $\sigma$-finite and not finite.  In this case, the ordinary component has an infinite number of atoms with probability one.

The ordinary component $\ord \rH$ of the directing random hazard measure $\rH$ can be related to the \as\ limiting frequencies $\process P \Nats$ of the underlying random partition.
\begin{thm}\label{simplestickbc}
Let $A \in \bsa$ such that $\gamma \defas \NABM(A) < \infty$. Then
\[
\ord \rH( \argdot \cap A) = \sum_{t=1}^\infty \sum_{j=1}^{\rcard_t} P_{tj} \, \delta_{s_{tj}}  \ \as,
\]
for some independent processes
$\nprocess {\rcard} \Nats t$, $\nprocess P {\Nats^2} {tj}$, and $\nprocess s {\Nats^2} {tj}$ such that:
\begin{enumerate}
\item
$\nprocess {\rcard} \Nats t$ are \iid\ Poisson random variables with mean $\gamma$;

\item
$\nprocess {{P_t}} \Nats {j}$, for $t \in \Nats$, are independent collections of \iid\ random variables such that $P_{tj} \equaldist P_t$;

\item
and $\nprocess s {\Nats^2} {tj}$ are \iid\ random variables with distribution $\gamma\inv \NABM(\argdot \cap A)$.
\end{enumerate}
\end{thm}

At this point, we can draw out several connections with well-known stochastic processes.  More details are given in \cref{sec:cup,sec:twoparam}:  
If $\Pi$ is a partition induced by a one-parameter CRP with concentration parameter $\conc$, then $\rH$ is a \emph{beta process}~\citep{MR1062708} with mean $\BM$.  In particular, $\Delta=0$, and so there is no \nonrandomcomp\ component; $\thekernel(\qq,\argdot) = \Beta(\conc\, \qq, \conc\, (1-\qq))$; and $\sdist = \Beta(1,\conc)$, and so, $\eta$ has intensity
\[
(\dee s, \dee \pp) \mapsto \NABM(\dee s) \, \conc\, \pp\inv (1-\pp)^{\conc-1} \,\dee \pp.
\]
It can be shown that
\[\label{stickbreakingrep}
P_n = V_n \prod_{j =1}^{n-1} (1-V_j) \AS,
\]
where $\process V \Nats$ are \iid\ and ${V_1 \dist \Beta(1, \conc)}$.
Combined with \cref{simplestickbc}, 
we arrive at the so-called stick-breaking construction of the beta process given by
Paisley, Zaas, Woods, Ginsburg, and Carin~\citep{PZWGC2010}.

If, on the other hand, $\Pi$ is a partition induced by a two-parameter CRP~\citep{MR1481784,MR515721},
with concentration parameter $\conc$ and discount parameter $\disc$, then $\Delta=0$; $\thekernel(\qq,\argdot) = \law(\sum_{i=1}^\infty P_i T_i)$, where $\process P \Nats$ is a size-biased permutation of the two-parameter Poisson Dirichlet and $\process T \Nats$ is an independent collection of \iid\ $\Bernoulli(\qq)$ random variables; and $\sdist = \Beta(1-\disc,\conc+\disc)$, and so, $\eta$ has intensity 
\[
(\dee s, \dee \pp) \mapsto 
\NABM(\dee s) \, \frac{ \Gamma(c+1) } { \Gamma(1-\disc)\, \Gamma(\conc + \disc) } \, \pp^{-1-\disc} (1-\pp)^{\conc+\disc-1} \, \dee \pp.
\]
The ordinary component is thus a \emph{stable beta process}, as defined by Teh and G\"or\"ur~\citep{TehGor2009a}.  
The \as\ limiting frequencies $\process P \Nats$ satisfy \cref{stickbreakingrep} but for merely independent random variables $V_n$, for $n \in \Nats$, where  $V_n \dist \Beta(1-\disc,\conc+n \disc)$.  Therefore,
\cref{simplestickbc} recovers the stick-breaking construction of the stable beta process given by Broderick, Jordan, and Pitman~\citep{BJP12}.  
(These authors refer to the same process as a three-parameter beta process.)

Even though Teh and G\"or\"ur did not define a \fixedcomp\ component, 
in our opinion, it would be natural to use the term ``stable beta process'' in order to refer to the class of random measures $\rH$ arising from two-parameter CRPs in this way.  In \cref{ssec:hsbp}, we discuss the omission of a \fixedcomp\ component in~\cite{TehGor2009a}, and the fact that our definition for the \fixedcomp\ component differs from that proposed by
Broderick, Mackey, Paisley, and Jordan~\citep{BJP11}.

In the case of beta processes and stable beta processes, 
one can give a characterization of the ordinary component as a countably-infinite sum of completely random measures, each finitely supported and accounting for the atoms appearing for the first time at each stage $X_1, X_2, \dotsc$ in a \COUS.  These representations have proven useful in applications in part because they are extremely simple to generate and yield finite approximation bounds. 

These constructions can be extended to \generalizedbetaprocess{}es:
For every $n \in \Nats$, let $\PFF {n}(\pp) = (1-\pp)^n$.  
We can then write $\PFF {n} \sdist$ for the Borel measure on $[0,1]$ given by
\[
(\PFF {n} \sdist) (B) = \int_{B} (1-\pp)^n\, \sdist(\dee \pp), \quad B \in \BorelSetsInt,
\]
We can also give a combinatorial interpretation of this measure:
Let $\nUni_n$  be the number of blocks in the random partition $\Pi$ restricted to $[n]$.
The event $\theset {\nUni_n > \nUni_{n-1}}$ is the event that $Z_n$ is a new token, and, on this event, $P_{\nUni_n}$ is the \as\ limiting frequency of appearance of this new token in the remainder of the sequence.
\cref{charofmu} shows that 
\[
(\PFF {n-1} \sdist)(B) = \Pr \theset { P_{\nUni_n} \in B \wedge \nUni_n > \nUni_{n-1} }
\]
for every $n \in \Nats$ and $B \in \BorelSetsInt$.

The identity, $\pp\inv = \sum_{n =1}^\infty (1-\pp)^{n-1}$, for $\pp \in (0,1)$,
yields the following construction of the ordinary component of generalized beta processes:
\begin{thm}\label{superposthm}
Let $\rH$ be the random \hazard\ measure directing a \COUS, and let $\eta$ be the Poisson process underlying the ordinary component as in \cref{genbeta}.   
Then
\[\label{superpos}
\sum_{(s,\pp) \in \eta} \!\!\! \pp\, \delta_s = \sum_{n=1 \vphantom{(}}^\infty \sum_{(s,\pp) \in \eta_n} \!\!\! \pp\, \delta_s \quad \as
\]
for some collection $\eta_n$, for $n \in \Nats$, of independent Poisson processes on $\bspace \times (0,1]$ with intensities
\[
(\EE \eta_n) (\dee s, \dee \pp) \mapsto \NABM(\dee s) \, \PFF {{n}-1} \sdist (\dee \pp).
\]
Note that for every $A \in \bsa$ such that $\NABM(A) < \infty$, we have $\eta_n(A\times(0,1]) < \infty$ a.s.
\end{thm}

The following result gives approximation error bounds when the above construction is truncated at a finite stage:
\begin{thm}\label{introapprox}
Assume $\gamma \defas \NABM(\bspace) < \infty$, 
and let $\rH$ be the random \hazard\ measure directing a \COUS\ $\process X \Nats$.
Let 
\[
\hat \rH \defas \Delta \NABM + \sum_{m=1 \vphantom{(}}^{k-1} \sum_{(s,\pp) \in \eta_m} \!\!\! \pp\, \delta_s 
\]
be the finite truncation of $\rH$, i.e., the sum of only the first $k-1$ terms of the right hand side of \cref{superpos},
and let $\hat X_1$ be the restriction of $X_1$ to the complement of the support of $\rH - \hat \rH$.
Then the expected total mass of the ordinary component of $\rH-\hat \rH$, and equivalently, an upper bound on the probability that $\hat X_1 \neq X_1$, is
$\gamma \, (\PFF {k-1} \sdist)(0,1]$.
\end{thm}

\newcommand{\Histories}{\mathcal H}

\subsection{The underlying combinatorial stochastic process}

In applications to latent feature models and the theory of exchangeable feature allocations, the combinatorial structure of a \COUS\ is of primary interest.
Let $\process X \Nats$ be a homogeneous \COUS.
For $n \in \Nats$ and $h \in \Histories_n \defas \theset{0,1}^n \setminus \theset {0^n}$, let $s(h) \defas \sum_j h(j)$ denote the number of nonzero entries and let $M_h$ be the number of elements $s$ such that $(\forall j \le n)\ X_j \theset s = h(j)$.  
For every (Borel) automorphism $\phi$ on $\bspace$, we can define the transformed processes $X'_n \defas X_n \circ \phi^{-1}$, for $n \in \Nats$, where each atom $s$ is repositioned to $\phi(s)$.  Note that the counts $M_h$, for $h \in \Histories_n$ are invariant to this transformation, and it is in this sense that they capture only the combinatorial structure.
Let $\CombStruct{ (X_1,\dotsc,X_n) }$ denote $\theset{M_h \st h \in \Histories_n}$.
In \cref{sec:combstruct}, we prove the following:

\begin{thm}\label{introcombstruct}
Let $\BM$ be nonatomic and finite, let 
$\gamma \defas \BM(\bspace)  < \infty$. Then
\[
\begin{split}
&\Pr \{ \CombStruct{ (X_1,\dotsc,X_n)} = (m_h \st h \in \Histories_n) \} 
\\&\qquad = 
\gamma^{\sum_{h \in \Histories_n} \! m_h } \, 
  \exp \oooooF { - \gamma \sum_{j=1}^n f(j,1) } 
 \prod_{h \in \Histories_n} \frac { f \ooof {n,s(h) }^{m_h} } {(m_h)! }
\end{split}
\]
where
\[\label{fnkdef}
f(n,k) \defas \int_{[0,1]} \pp^{k-1} (1-\pp)^{n-k} \, \sdist(\dee \pp), \qquad \text{ for $k \le n \in \Nats$.}
\]
\end{thm}

\begin{remark}
The following identities relate $f(n,k)$ to combinatorial events in the underlying exchangeable partition:
Let $h \in \Histories_n$ such that $s(h)=k$,
let $\Pi_n$ be the restriction of $\Pi$ to $[n]$,
and 
recall the definition of $N_{jn}$ above.
Then, by exchangeability,
\[
f(n,k) 
&= \Pr \theset { h\inv (1) \in \Pi_n } \\
&= \Pr \theset { N_{1k} = k \wedge N_{1n}=k } = \Pr \theset { N_{ {\nUni_{n-k+1}}n} = k }  \\
&= { n-1 \choose k-1 }\inv \, \Pr \theset { N_{1n} = k } = { n-1 \choose k-1 }\inv \, \Pr \theset { N_{\nUni_n n} = k }.
\]
These identities may be simpler to work with than \cref{fnkdef}.
\end{remark}

It is worth pausing to highlight the connection with the IBP:
If $\Pi$ is a partition induced by a one-parameter CRP with concentration parameter $\conc$, then 
$\sdist = \Beta(1,\conc)$, and so, $f(n,k) = c \frac { \Gamma (k) \Gamma( c+n-k )} {\Gamma(c+n)}$.  The resulting p.m.f.\ is then precisely that of the two-parameter IBP \citep{GGS2007}, with concentration $\conc$ and mass $\gamma$.  Taking $\conc=1$, one recovers the original IBP, proposed by \citep{GG05,GG06}.
If, on the other hand, $\Pi$ is a partition induced by a two-parameter CRP, 
with concentration parameter $\conc$ and discount parameter $\disc$, one recovers the three-parameter IBP proposed by 
Teh and G\"or\"ur \citep{TehGor2009a}. 

\bigskip

The organization of the remainder of the article is as follows: 
In \cref{sec:oneparam}, we define a one-parameter scheme and show that it is an exchangeable sequence of Bernoulli processes directed by a beta process.  It follows that its combinatorial structure is an IBP.  But more importantly, the combinatorial structure of a hierarchy of one-parameter schemes, which corresponds to a hierarchy of beta processes, is seen to be the missing hierarchical version of the IBP.
In \cref{sec:esp}, we introduce some necessary preliminaries on exchangeable sequences and their directing random measures.
In \cref{sec:cup}, we define the \COUS\ with respect to a measurable family of EPPFs, show that the resulting sequence of simple point processes is exchangeable, and indeed corresponds with an exchangeable sequence of Bernoulli processes directed by a generalization of the beta process, whose ordinary, \fixedcomp\ and \nonrandomcomp\ (due to dust) components we characterize in terms of the EPPFs.
We end the section by describing the IBP analog.
In  \cref{sec:twoparam}, we consider the EPPF corresponding with the two-parameter Chinese restaurant process, producing a two-parameter \COUS\ that we show corresponds to an exchangeable sequence of Bernoulli processes directed by a generalization of the stable-beta process.   The combinatorial process is shown to be the three-parameter IBP introduced by Teh and G\"or\"ur.
Finally, in \cref{sec:climit}, we return to the limiting construction alluded to in this introduction, and show that a general \COUS\ can be obtained as a weak limit of finite processes.

\newcommand{\Levy}{L\'evy}
\section{Preliminaries} %
\label{sec:background}

In this section we very briefly review some definitions and results from the theory of completely random measures; define beta and Bernoulli processes; and develop a few additional properties of Bernoulli processes.

\subsection{Random measures; point processes; Poisson processes}

We fix a basic probability space $\probspace$ which one can assume to be the unit interval, equipped with the $\sigma$-algebra of Lebesgue-measurable sets, under Lebesgue measure.
The following setup is (essentially) taken from Kallenberg \citep[][Chp.~12]{FMP2}.  We reproduce it here 
(indented, with several small modifications) 
for completeness.

\newenvironment{borrowedtext}{\list{}{\leftmargin=2.5mm\rightmargin=2.5mm}\item[]}{\endlist}
\begin{borrowedtext}
Consider an arbitrary measurable space $\borelspace$:
we say that $\xi$ is a \defn{random measure} on $\borelspace$ when it is a $\sigma$-finite kernel from the basic probability space $\probspace$ into $\bspace$, 
i.e., $\xi$ is a map from $\uspace \times \bsa$ to $[0,1]$ such that:
\begin{enumerate}
\item $\xi(\upoint,\argdot)$ is a measure for all $\upoint \in \uspace$;
\item $\xi A \defas \xi(\argdot,A)$ is a random variable for all $A \in \bsa$; and
\item for some partition $A_1,A_2,\dotsc \in \bsa$ of $\bspace$, $\xi A_k < \infty$ \as\ for all $k$.
\end{enumerate}
(Note that we require a single partition to be witness to the $\sigma$-finiteness of $\xi$.)
We may consider $\xi$ to be a random element in the space $\Measures\borelspace$ of $\sigma$-finite measures on $\borelspace$, equipped with the $\sigma$-algebra generated by the projection maps $\pi_{A} \st \mu \mapsto \mu(A)$ for all $A \in \bsa$.
We define the \defn{intensity} of $\xi$ to be the measure $\EE\xi$ given by $(\EE\xi) A = \EE(\xi A)$, for $A \in \bsa$.

We say that $\xi$ is a \defn{point process} when it is an integer-valued random measure, i.e., when $\xi A$ is a $\NNExtInts$-valued random variable for every $A \in \bsa$. 
Alternatively, we may think of $\xi$ as a random element in the space of all $\sigma$-finite, integer-valued measures on $\bspace$.  When $\bspace$ is Borel (e.g., when $\bspace$ is lcscH), we may write $\xi = \sum_{k \le \rcard} \delta_{\gamma_k}$ for some random elements $\gamma_1,\gamma_2,\dotsc$ in $\bspace$ and $\rcard$ in $\NNExtInts$, and we note that $\xi$ is \defn{simple} iff the $\gamma_k$ with $k \le \rcard$ are distinct.  (We will say that a measurable space $\borelspace$ is \defn{Borel} if there exists a measurable bijection $\phi$ from $\borelspace$ onto a Borel subset of $\Reals$, whose inverse is also measurable.)  In general we may eliminate the possible multiplicities in $\gamma_1,\gamma_2,\dotsc$ to create a simple point process $\xi^*$, which agrees with the counting measure on the support of $\xi$.  By construction it is clear that $\xi^*$ is a measurable function of $\xi$.

A random measure $\xi$ on a measurable space $\borelspace$ is said to have \defn{independent increments} if the random variables $\xi \bset_1, \dotsc, \xi \bset_n$ are independent for any disjoint sets $\bset_1,\dotsc,\bset_n \in \bsa$.  By a \defn{Poisson process} on $\bspace$ with intensity measure $\mu \in \Measures\borelspace$ we mean a point process $\xi$ on $\bspace$ with independent increments such that $\xi \bset$ is Poisson with mean $\mu \bset$ whenever $\mu \bset < \infty$.  These conditions specify the distribution of $\xi$, which is then determined by the intensity measure $\mu$. 

\end{borrowedtext}

\subsection{Completely random measures}

The law of a random measure $N$ on $\bspace$ is uniquely characterized by its \emph{characteristic functional}
\[
f \mapsto
\EE [ e^{- Nf} ],  \qquad \text{ $f \colon \bspace \to \NNReals$, measurable,}
\]
where $Nf \defas \int f(s)\, N(\dee s)$. 
The following result characterizes Poisson processes:
\begin{thm}[Campbell]\label{campbells}
Let $N$ be a Poisson process on $\borelspace$ with nonatomic mean measure $\mu$ and let $f\colon \bspace \to \NNReals$ be measurable.  Then
\[
\EE[e^{-Nf }]
= \exp \Bigl [ - \int (1 - e^{- f(s)} )\, \mu(\dee s) \Bigr ] \, . 
\]
\end{thm}

By a \defn{completely random measure} we simply mean a random measure with independent increments.  
Poisson processes are the simplest type of completely random, and as we will see, Poisson processes play a fundamental role  in the theory of completely random measures.
(The interested reader is referred to \citep{MR0210185} and Kallenberg~\citep[][Chp.~12]{FMP2} for further details.  
Note that our definition of random measure ensures certain weak finiteness conditions.)

Let $\xi$ be a random measure on a \general\ space $\borelspace$.  We say that $\bp \in \bspace$ is a \defn{fixed atom} when $\Pr \theset{ \xi\theset{\bp}  > 0 } > 0$.
We begin with a characterization of completely random measures without fixed atoms:

\begin{thm}[{Kingman~\citep{MR0210185}, Kallenberg~\citep[][Cor.~12.11]{FMP2}}]\label{thm:crmrep}
Let $\xi$ be a random measure on $\borelspace$ such that $\xi\theset s=0$ \as\ for all $s$.
Then $\xi$ is a completely random measure if and only if
\[\label{crmrep}
\xi A = \mu A + \int_{(0,\infty)} \pp \, \eta(A \times \dee \pp) \AS,\qquad  A \in \bsa,
\]
for some nonrandom nonatomic measure $\mu$ on $\bspace$ and Poisson process $\eta$ on $\bspace \times (0,\infty)$.
Furthermore, $\xi A < \infty$ \as\ for some $A \in \bsa$ if and only if $\mu(A) < \infty$ and 
\[
\int_0^\infty (\pp\wedge 1) \, \EE \eta(A \times \dee \pp) < \infty.
\]
Note that $\EE \eta\{(s,\pp) \st \pp > 0\} = \EE \eta (\bspace \times (0,\infty)) = 0$ for all $s$.
\end{thm}

\begin{remark}
We will sometimes write \cref{crmrep} more compactly as
\[
\xi = \mu + \sum_{(s,\pp) \in \eta} \pp\, \delta_s \AS
\]
\end{remark}

When a completely random measure $\xi$ has the form of \cref{crmrep}, we call $\mu$ the 
\defn{\nonrandomcomp\ component}, and we will call $\xi - \mu$ the \defn{ordinary component}.   It follows immediately from \cref{thm:crmrep} that an arbitrary completely random measure is of the form $\xi + \chi$, where $\xi$ is as in \cref{thm:crmrep} and $\chi$ is a purely-atomic random measure, independent of the Poisson process $\eta$, and supported on a nonrandom countable subset $\Atoms \subseteq \bspace$, where the random variables $\chi\theset s$, for $s\in \Atoms$, are independent.  In this case, we will call $\chi$ the \defn{\fixedcomp\ component} and $\Atoms$ the \defn{fixed atoms}.

Consider the so-called \defn{L\'evy measure} on $\bspace \times (0,\infty)$ given by
\[
\nu \defas \EE \of { \textstyle\sum_s \delta_{(s, \xi\theset s)} }.
\]
It follows that 
\[\textstyle
\nu = \EE \eta + \sum_{s \in \FixedAtoms} \delta_{s} \otimes \law(\xi \theset s)
\]
and 
$\nu(\theset s \times (0,\infty)) \le 1$ for all $s \in \bspace$,
and so the law of a completely random measure $\xi + \chi$ is uniquely characterized by specifying 
its \nonrandomcomp\ component $\mu$ and its L\'evy measure $\nu$, as the latter encodes the position and distribution of the fixed atoms, as well as the intensity of the underlying Poisson process.

In the other direction, given any $\sigma$-finite measure $\nu$ on $\bspace \times (0,\infty)$ such that 
$\nu( \theset s \times (0, \infty)) \le 1$, for every $s \in \bspace$,
there is a completely random measure whose L\'evy measure is $\nu$.
In particular, let 
\[
\Atoms \defas \{ s \in \bspace \st \nu(\theset s \times (0,\infty)) > 0 \}
\]
be the countable set of nonnull $\bspace$-sections of $\nu$, let $\bar\eta$ be the restriction of $\nu$ to $(\bspace \setminus \Atoms) \times (0,\infty)$, let $\eta$ be a Poisson process with intensity $\bar\eta$, and let $\chi$ be a random measure independent from $\eta$ and supported on $\Atoms$ such that the masses $\chi \theset s$, for $s \in \Atoms$, are independent with distribution $\nu( \theset s \times \argdot) + (1-\nu( \theset s \times (0,\infty)))\, \delta_0$.  Then 
\[
\chi + \sum_{(s,\pp) \in \eta} \pp\, \delta_s 
\]
is completely random, with L\'evy measure $\nu$.

If we write $M_s (t) \defas \EE [ e^ { -t \, \chi \theset s }] $ for the moment-generating function of $-\chi \theset s$, for $s \in \FixedAtoms$, then the characteristic functional of $\xi + \chi$ is the map 
\[
f \mapsto 
\exp \Of {  -\mu f + \sum_{s \in \FixedAtoms} M_s(f(s))
- \int (1- e^{-\pp\, f(s)}) \, \EE \eta(\dee s \times \dee \pp)
}.
\]

We now introduce two families of completely random measures---beta and Bernoulli processes---that are the focus of \cref{sec:oneparam}.
The beta process was introduced by Hjort~\citep{MR1062708} and later connected to Indian buffet process by Thibaux and Jordan~\citep{Thibaux2007}, who introduced the notion of Bernoulli processes as defined below.

\begin{quote}
{\bf 
For the remainder,
let 
\[
\BM \defas \NABM + \textstyle\sum_{s \in \FixedAtoms} \BM\theset s \,\delta_s
\]
be a $\sigma$-finite measure on $\bspace$, where $\NABM$ is nonatomic; $\FixedAtoms \subseteq \bspace$ is countable; and $\BM \theset s \in (0,1]$ for all $s \in \FixedAtoms$.
}
\end{quote}

\subsection{Bernoulli processes}

By a \defn{Bernoulli process} with \bermean\ $\BM$
we mean a purely-atomic completely random measure $X$ with L\'evy measure $\BM \otimes \delta_1$, 
written
\[
X \dist \BePLAW(\BM).
\]
In particular, $X$ has no \nonrandomcomp\ component, its fixed atoms are the atoms $\FixedAtoms$ of $\BM$, each appearing independently in $X$ with probability $\BM\theset s$, and the intensity measure of the Poisson process underlying the ordinary component $X$ is $(\dee s,\dee q) \mapsto \delta_1(\dee q) \, \NABM(\dee s)$.
The characteristic functional of $X$ is
\[
f \mapsto
\exp \Bigr [ - \int (1- e^{-f(s)}) \, \NABM(\dee s) \Bigr ] \times 
\prod_{s \in \FixedAtoms} \Bigr [ 1 - \BM\theset s + \BM\theset s \, e^{-f(s)} \Bigr ],
\]
which, in light of \cref{campbells}, highlights the relationship between Bernoulli processes and Poisson processes.  In particular, the ordinary component of $X$ is simply a Poisson process with intensity $\NABM$.

By the form of the L\'evy measure, it is straightforward to show that a Bernoulli process is \as\ simple.
In fact, every \as\ simple completely random measure is a Bernoulli process:

\begin{thm}\label{thm:bernoullichar}
Let $X$ be a random measure on a \general\ space $\borelspace$.  Then $X$ is a Bernoulli process if and only if $X$ is \as\ simple and completely random.
\end{thm}
\begin{proof}%
The forward direction follows in a straightforward way from the definition of a Bernoulli process.  
In the other direction, let $X$ be \as\ simple and completely random, and put $\BM = \EE X$.  Then $X = \chi + \xi$ is the sum of a \fixedcomp\ component $\chi$ and ordinary component $\xi$.  By the \as\ simplicity of $\chi$, we may write $\chi = \sum_{s \in \FixedAtoms} \pp_s \delta_s$ where $\FixedAtoms \subseteq \bspace$ is countable, and the $\pp_s$, for $s\in\FixedAtoms$, are independent Bernoulli random variables with mean $\BM\theset s$.  It follows that the L\'evy measure of $\chi$ is $\sum_{s\in\FixedAtoms} p_s \delta_{s,1} = \BM(\cdot \cap \FixedAtoms) \otimes \delta_1$ and so $\chi$ is a Bernoulli process.  By independence of increments, it suffices to show that $\xi$ is also a Bernoulli process.

We have $\xi\theset s= 0$ \as\ for all $s$, and so 
$\eta 
 = \sum_s \delta_{s,\xi\theset s}
 = \sum_{s \in \xi} \delta_{s,1}$ 
is a completely random point process on $\bspace \times (0,\infty)$ satisfying $\eta(\theset s \times (0,\infty)) = 0$ \as\ for all $s$, and moreover, $\eta( \cdot \times (0,\infty))$ is a $\sigma$-finite point process.
It follows from Theorem 12.10 (FMP) that $\eta$ is a Poisson process.  But then $\xi = \sum_{(s,\pp) \in \eta} \pp\, \delta_s$, and so $\xi$ has L\'evy measure $\EE\eta = \BM(\cdot \setminus \FixedAtoms) \otimes \delta_1$, and is thus a Bernoulli process.
\end{proof}

It follows immediately that the law of a Bernoulli process $X$ is characterized by its mean $\EE X = \BM$.  (If the so-called \defn{mass parameter} $\BM(\bspace)$ is finite, then it is also the expected cardinality of the Bernoulli process when considered as a random set.)

\subsection{Beta processes}
Let $\conc : \bspace \to \NNReals$ be a measurable function.
By a \defn{beta process with concentration function $\conc$ and \bpmean\ $\BM$},
we mean a completely random measure $\rH$ on $\borelspace$, written
\[
\rH \dist \BPLAW(\conc,\BM),
\]
when it is a purely-atomic completely random measure with L\'evy measure $\nu_{na} + \nu_{a}$ where 
\[
\nu_{na}(\dee s,\dee p) = \conc(s) \,p^{-1}(1-p)^{\conc(s)-1} \dee p \, \NABM(\dee s)
\]
corresponds with the ordinary component and
\[
\nu_{a} = \sum_{s \in \FixedAtoms} \delta_s \otimes \Beta \ooof{ \conc(s) \, \BM \theset s, \conc(s)\,(1-\BM \theset s) }
\]
corresponds with the \fixedcomp\ component.  (Implicit is the requirement that $\nu_{na}$ is $\sigma$-finite, which follows, e.g., if $\theta$ is bounded, because $\NABM$ is assumed to be $\sigma$-finite.) 
When $\conc$ is constant, we will refer to it as the concentration \emph{parameter}.  

The remainder of the document will provide a great deal of insight into the structure of a beta process, but it is worthwhile stating a few of its properties here:  
First of all, as expected, $\EE B = \BM$.
When $\NABM$ is nonzero, $\nu_{na}$ is merely $\sigma$-finite, even when $\NABM$ is finite, and so the ordinary component $\ord \rH$ of $\rH$ has infinitely many atoms with probability one. 
The beta process has several direct ``stick-breaking'' constructions.  
Several of these (\citep{Thibaux2007,PZWGC2010,PBJ2012}) have analogues in \cref{sec:intro} and later, and so we describe one due to 
Teh, G\"or\"ur, and Ghahramani \citep{TehGorGha2007}, which is particularly simple to describe.  In the case when $\gamma = \NABM (\bspace) < \infty$ and $\theta \equiv 1$, 
\[
\ord \rH = \sum_{n=1}^\infty \Ooooof { \prod_{j=1}^n V_j } \, \delta_{\varsigma_n} \AS,
\]
where $\process \varsigma \Nats$ are \iid\ and $\gamma\inv \NABM$-distributed, and $\nprocess V \Nats j$ are \iid\ and $\Beta(\gamma,1)$-distributed.

Finally, the beta and Bernoulli process are conjugate in the following sense:
\begin{thm}[conjugacy; Hjort, Kim, Thibaux-Jordan] \label{berconditional}
Let $\rH$ be a beta process on $\bspace$ with \bpmean\ $\BM$ and concentration parameter $\conc>0$.  Conditioned on $\rH$,  let $\process X \Nats$ be independent Bernoulli processes with \bermean\ $\rH$.  Then
\[\label{postpred}
X_{n+1} \given X_{[n]} \sim \BePLAW \oooof{ \frac {\conc} {\conc+n} \BM + \frac 1 {\conc+n} \sum_{i=1}^n X_i }.
\]
\end{thm}

\begin{remark}
This result was first shown by Hjort~\citep[][Cor.~4.1]{MR1062708} for the case of censored observations in $\bspace = \Reals$.  This result can be seen as a corollary of a result due to Kim~\citep[][Thm.~3.3]{MR1714717}, who studied censored observations from general completely random hazard measures.   Thibaux and Jordan~\citep{Thibaux2007} presented the result in the form above, and showed that an Indian buffet process was the combinatorial structure of a conditionally \iid\ sequence of Bernoulli processes satisfying \cref{postpred}.  Note that Theorem~3.3 of \citep{MR1714717} assumes that the \nonrandomcomp\ part of $\BM$ is absolutely continuous w.r.t.\ Lebesgue measure.  This is not necessary; indeed, the proof does not rely on the assumption in any deep way.  
\cref{maindefinetti} implies the above claim with no such assumption, and so we omit the proof here.
\end{remark}

The conditional independence and mean structure above will reappear many times below, and so we introduce the following terminology:

\begin{definition}
Let $\rH$ be a random measure and let $\process X \Nats$ be a sequence of Bernoulli processes.  We will say that 
$\process X \Nats$ is \defn{an exchangeable sequence of Bernoulli processes directed by $\rH$} (or, when the context is clear, that
$\process X \Nats$ is \defn{directed by $\rH$})
when, conditioned on $\rH$, the $X_1,X_2,\dotsc$ are independent Bernoulli processes with \bermean\ $\rH$.
\end{definition}

For every measurable function $f : \bspace \to \NNReals$ and measure $\nu$ on $\borelspace$, define the measure $f\nu$ by $(f\nu)(A) = \int_A f(s) \nu(\dee s)$, for $A\in\bsa$.  The following result will be used regularly without comment, and follows easily from an approximation by simple functions and monotonic convergence:

\begin{prop}\label{prop:rdmean}
Let $f : \bspace \to \NNReals$ be a nonnegative, measurable function and $\xi$ a random measure on $\borelspace$ with intensity $\rH$.
Then $\EE f\xi= f \rH$.\qed
\end{prop}

\section{The continuum-of-Blackwell--MacQueen-urns scheme}
\label{sec:oneparam}

Let $\conc \ge 0$, let
$\Uniform$ denote the uniform distribution on $[0,1]$,
and let $\process Z \Nats \defas (Z_1,Z_2,\dotsc)$ be a sequence of random variables in $[0,1]$ such that $Z_1 \dist \Uniform$ and %
\[\label{bmus}
Z_{n+1} \given Z_{[n]} \dist \frac {\conc} {\conc+n}\, \Uniform + \frac 1 {\conc + n} \sum_{j \le n} \delta_{Z_j}, \quad \text{for } n \in \Nats,
\]
where $Z_{[n]} \defas (Z_1,\dotsc,Z_n)$.
In other words, $\process Z \Nats$ is a \emph{Blackwell-MacQueen urn scheme}, i.e., a conditionally \iid\ sequence of random variables directed by a Dirichlet process (with mean $\Uniform$ and concentration parameter $\conc$, in this case).  The combinatorial structure of $\process Z \Nats$, i.e., the random partition of $\Nats$ induced by the random equivalence relation $\{ (n,m) \subseteq \Nats\times\Nats \st Z_n = Z_m \}$, is that of a Chinese restaurant process, in which the probability of a new table is proportional to $\conc$.

The focus of the remainder of the article is the following construction and its generalizations:
Let $\nY \defas \process \nY \Nats$ be a nonrandom sequence of 
simple measures on a \general\ space $\borelspace$ concentrated on a locally-finite countable set $\Atoms$, and, for every $s \in \Atoms$, let
$\process {Z^s} \Nats$ be an independent copy of $\process Z \Nats$.
Consider the sequence $\process {\nX} \Nats$ 
of simple point processes on $\bspace$, concentrated on $\Atoms$, 
where $\nX_1 \defas \nY_1$ and, for every $n \in \Nats$ and $s\in\Atoms$, we define
\[
\nX_{n+1} \theset s \defas
\begin{cases}
\nX_j \theset s & \text{if $Z^s_{n+1} = Z^s_j$, where $j \le n$,} \\
\nY_{n+1} \theset s & \text{otherwise},
\end{cases}
\]
and, for every $A \in \bsa$, define $\nX_{n+1} A \defas \sum_{s \in A\cap\Atoms} \nX_{n+1} \theset s$.

In other words, $\nX_{n+1}$ is a simple random measure, whose atoms are some random subset of the atoms among $\nY_1,\dotsc,\nY_{n+1}$, and, in particular, conditioned on $\nX_{[n]}$, each such atom $s$ is an atom of $\nX_{n+1}$ independently with probability 
\[\label{oneurn}
\frac {\conc} {\conc+n} \nY_{n+1}\theset s + \frac {1}{\conc+n} \sum_{j=1}^n \nX_j\theset s.
\]

It is straightforward to show that $\process {\nX} \Nats$ is well-defined and that its law, which we will denote by $\cQ {\nY}$, is a measurable function of $\nY$.

\begin{definition}[one-parameter scheme]
Let $\BM$ be a \hazard\ measure on $\borelspace$, and 
let $X \defas \process X \Nats$ be a sequence of random measures on $\borelspace$.
Then we will say that $X$ is a \defn{one-parameter scheme with \cupmean\ $\BM$ and concentration parameter $\conc$} when 
there exists an i.i.d.\ sequence $Y \defas \process Y \Nats$ of Bernoulli processes with \hazard\ measure $\BM$ such that
$\Pr[X|Y] = \cQ {Y}$ a.s.  
\end{definition}

\begin{remark}
We have defined a one-parameter scheme in this way so that the relationship to the more general \COUS\ defined in \cref{sec:cup} is manifest.\footnote{We could have constructed $\process X \Nats$ directly from a sequence of simple point processes $\process Y \Nats$, although doing this rigorously is somewhat cumbersome. 
This construction, based on a randomization, sidesteps several measure-theoretic complications. (See~\citep[][Pg.~226]{FMP2} for another example of randomization.)  
A nonstandard, though elegant construction would employ an \iid\ collection $\process {Z^s} \Nats$, for $s \in \bspace$, of urn schemes, but working with a continuum of \iid\ random variables leads quickly to measurability roadblocks unless one, e.g., works on the minimal (and unique~\citep{MR2180892}) extension of the basic probability space that made $Z$ and $X$ jointly measurable.  On this extended space, we would have a one-way Fubini property, which would allow us to show that the joint law of $\process Y \Nats$ and $\process X \Nats$ is precisely as described above.  Arguably, a construction using such an \iid\ process is more aptly named a \COUS, but we have decided to give a standard construction.}
Because of the special properties of the Blackwell-MacQueen urn scheme, 
we will see that the law of a one-parameter scheme has a simple conditional characterization, which could equally well have served as the definition.
\end{remark}

Let $\cG_n = \sigma(Y_{n+1},X_{[n]})$.
From \cref{oneurn}, we may conclude that $X_{n+1}$ is \as\ simple and has conditionally independent increments given $\cG_n$. 
Therefore, conditioned on $\cG_n$, by \cref{thm:bernoullichar}, $X_{n+1}$ is a Bernoulli process with \bermean\ given by
\[
\EE^\cG X_{n+1} = \frac {\conc} {\conc+n} Y_{n+1} + \frac {1}{\conc+n} \sum_{j=1}^n X_j \ \as,
\]
Because $Y$ is an \iid\ sequence, we may conclude from the definition of $\cQkernel$ that $Y_{n+1}$ is independent of $\cF_n \defas \sigma(X_{[n]}) \subseteq \cG$, and so, by the chain rule of conditional expectation,
\[\label{aoseuheu}
\EE^{\cF_n} X_{n+1} = \EE^{\cF_n} [\EE^{\cG_n} X_{n+1}] = \frac {\conc} {\conc+n} \BM + \frac {1}{\conc+n} \sum_{j=1}^n X_j \ \as
\]
Moreover, $X_{n+1}$ has conditionally independent increments given $\cF_n$ because $Y_{n+1}$ does.
It follows that 
\[\label{aoseuheueue}
X_{n+1} | \cF_n \dist \BePLAW(\EE^{\cF_n} X_{n+1}).
\]
From \cref{berconditional},
we can now recognize $\process X \Nats$ as having the same conditional law as a conditionally \iid\ sequence of Bernoulli processes directed by a beta process.

\begin{thm}[de~Finetti measure]\label{thm:main}
Let $\process X \Nats$ be a one-parameter scheme on $\borelspace$ with 
\cupmean\ $\BM$ and concentration parameter $\conc$.  Then there is an \as\ unique random measure $\rH$ given by
\[\label{drhmconvergence}
\rH(A) = \lim_{n\to\infty} n^{-1} \sum_{i=1}^n X_i (A) \ \as, \quad A \in \bsa,
\]
Moreover, $\rH$ is a beta process with \bpmean\ $\BM$ and concentration parameter $\conc$, 
and, conditioned on $\rH$, the random measures $X_1,X_2,\dotsc$ are independent Bernoulli processes with \bermean\ $\rH$.
\end{thm}

\begin{proof}%
By \cref{aoseuheu,aoseuheueue}, for all $n \in \NNInts$, the conditional distribution of $X_{n+1}$ given $X_{[n]}$ agrees with \cref{postpred} in \cref{berconditional} when we take $\BM \defas \EE Y_1$.  As the law of stochastic process is determined by its finite dimensional distributions, it follows that
$\process X \Nats$ has the same law as 
an exchangeable sequence $ \process {X'} \Nats$ of Bernoulli processes directed by a beta process \cupmean\ $\rH'$ with \bpmean\ $\BM$ and concentration $\conc$.  

By a transfer argument (\cref{transfer}), there exists a random measure $\rH$ such that 
$ %
(\rH,X_1,X_2,\dotsc) \equaldist (\rH',X'_1,X'_2,\dotsc),
$ %
and so
$\rH$ renders $\process X \Nats$ conditionally independent;
$\rH\equaldist \rH'$ and in particular $\rH$ is a beta process with \hazard\ measure $\BM$ and concentration $\conc$; 
and 
conditioned on $\rH$, each $X_n$ is a Bernoulli processes with \hazard\ measure $\rH$.
Finally, letting $\mathcal F = \sigma(\rH)$, we have $\EE^{\mathcal F} X_n = \rH$ and therefore, we may conclude from the disintegration theorem (\citep[][Thm.~6.4]{FMP2})
and the law of large numbers, that \cref{drhmconvergence} holds.  Because $\bspace$ is Polish, there is a countable $\pi$-system that generates $\bsa$, and so this convergence holds simultaneously for the $\pi$-system and thus $\rH$ is \as\ unique by \citep[][Lem.~1.17]{FMP2}.
\end{proof}

As one can anticipate from the work of Thibaux and Jordan~\citep{Thibaux2007},
the one-parameter scheme is related to the Indian buffet process (IBP) introduced by Griffiths and Ghahramani~\citep{GG05}.
In order to make the connection precise, we introduce the following quotient of the space of sequences of simple measures:
for any pair $U \defas(U_1,U_2,\dotsc)$ and $V \defas (V_1,V_2,\dotsc)$ of sequences of simple measures, $U \sim V$ when there exists a Borel isomorphism 
$\varphi : \bspace \to \bspace$ satisfying $U_n = V_n \circ \varphi^{-1}$ for every $n$.  It is easy to verify that $\sim$ is an equivalence relation.  Let $\CombStruct{U}$ denote the equivalence class containing $U$.  This quotient space is itself a Polish space, and can be related to the Polish space of sequences of simple measures by coarsening the $\sigma$-algebra to that generated by the functionals 
\[
\varphi_h (U_1,U_2,\dotsc) \defas \card {\theset { s \in \bspace \st (\forall j \le n)\, U_j\theset s = h(j) }}, \quad h \in \theset{0,1}^n,\, n\in\Nats.
\]
In other words, $\CombStruct{U}$ maintains only enough information to count, for every $n \in \Nats$, and subset $S \subseteq \Nats$, how many points $s \in \bspace$ are atoms of every and only those sets $X_j$, for $j \le S$, among $X_{[n]}$.  (In the author's opinion, this is the more natural characterization of the equivalence classes induced by the so-called \emph{left-ordered form} proposed by Griffiths and Ghahramani~\citep{GG05}.)

The following connection with IBPs follows both from \cref{thm:main} and~\citep{Thibaux2007}, as well as from the general result (\cref{cupibp}) for the \COUS:
\begin{thm}\label{thm:ibp}
Let $\conc>0$, let $\BM$ be a nonatomic \hazard\ measure such that $\gamma \defas \BM(\bspace) < \infty$,
and let $\process X \Nats$ be a one-parameter scheme induced by $\process Y \Nats$ with concentration parameter $\conc$.
Then $\CombStruct{\process X \Nats}$ is an IBP with mass parameter $\gamma$ and concentration parameter $\conc$.
\end{thm}

The IBP was first defined for the case $\conc = 1$ by Griffiths and Ghahramani \citep{GG05,GG06}, and later generalized to $\conc > 0$ by 
Ghahramani, Griffiths, and Sollich \citep{GGS2007}.

\section{Exchangeable sequences and partitions}
\label{sec:esp}

In this section, we introduce some concepts relating to exchangeable sequences of random variables and partitions.  (The following development owes much to~\citep{MR1481784}, which also provides more details for the interested reader.)  These results will be used subsequently to introduce and characterize a generalization of the one-parameter scheme.

Let $\process Z \Nats$ be an exchangeable sequence of random variables taking values in a \general\ space, and assume that the marginal distribution of the first (and thus every) element, $\mdrm \defas \Pr \theset{ Z_1 \in \argdot }$, is nonatomic.
We are interested in the combinatorial structure of the sequence.  In particular,
let $\Pi_n$ and $\Pi$ be the random partition of $[n]$ and $\Nats$, respectively, induced by the equivalence relation
\[
i \sim j \iff Z_i = Z_j.
\]
We may then write $\Pi = \theset{C_1, C_2, \dotsc }$, where $C_1$ is the class containing $1$, and $C_{n+1}$ is the class containing the first element of $\Nats \setminus \bigcup_{i \le n} C_i$, provided such an element exists, and is the empty class otherwise. (I.e., we define $C_{i} \defas \emptyset$ when $\Pi$ contains fewer than $i$ classes.)  

\newcommand{\ft}{\mathsf M}
To complete the decomposition of $\process Z \Nats$,  define 
\[\label{defnofM}
\ft_i \defas \inf C_i, \quad \text{for $i \in \Nats$,} 
\]
with the convention that $\inf \emptyset = \infty$.  On the event that $\ft_i < \infty$, i.e., $C_i$ is nonempty,  define $\tilde Z_i \defas Z_{\ft_i}$.  We say that $\tilde Z_i$ is the \defn{$i$-th token to appear}.  It is clear that the partition $\Pi$ and tokens $\tilde Z_i$ completely 
determine the sequence $\process Z \Nats$.

We proceed to characterize the probabilistic structure of the combinatorial part.
Let $\nUni_n \le n$ be the number of equivalence classes in $\Pi_n$, i.e., the number of unique elements among $\theset {Z_1,\dotsc,Z_n }$, and take $\nUni_0 \defas 0$.  
For every $j,n \in \Nats$, let 
$\nmul_{jn} \defas \card{ \theset{ i \le n \st Z_i = \tilde Z_j } }$ denote the multiplicity of the $j$-th token to appear among the first $n$ elements.
Then $\nmul_n \defas (\nmul_{1n}, \nmul_{2n}, \dotsc)$ is a vector of counts for each token, and is necessarily a sequence of $\nUni_n$ positive integers terminated by an infinite sequence of zeros, and so we may identify the range of these random vectors with the space $\Nats^* \defas \bigcup_{n \in \Nats} \Nats^n$ of finite sequences of positive integers in the obvious way.    
For a sequence $n =(n_1,\dotsc,n_k,0,0,\dotsc)$, let $n^{+j}$ be the sequence where $n_j$ is incremented by $1$.

For every finite sequence $(n_1, \dots, n_k) \in \Nats^*$, we may define 
\[
\pi(n_1,\dotsc,n_k) \defas \Pr \of { \nUni_n = k, \nmul_{1n}=n_1, \dotsc, \nmul_{kn} = n_k }.
\] 
By exchangeability, it follows that $\pi$ is a symmetric function.  By construction, %
\[\label{eppfprop}
\text{$\pi(1) = 1$ and $\pi(n) = \sum_{j=1}^{k+1} \pi(n^{+j})$ for every $n \in \Nats^*$.}
\]
A symmetric function on $\Nats^*$ satisfying~\cref{eppfprop} is known as an \defn{exchangeable partition probability function} (EPPF) and can be seen to completely characterize the distribution of the combinatorial structure of $\process Z \Nats$.  (See~\citep{MR1481784} for more details.)  
In particular, it can be shown that the conditional distribution of $Z_{n+1}$ given $Z_{[n]}$ is
\[\label{eppfpred}
\sum_{j=1}^{\nUni_n}
\frac {\eppf(\nmul_n^{+j})}
        {\eppf(\nmul_n)}
   \, \delta_{\tilde Z_j} 
+
\frac {\eppf(\nmul_{n}^{+(\nUni_n+1)})}
        {\eppf(\nmul_n)}
   \,\mdrm.
\]
Because of this underlying urn scheme structure,
we will refer to any exchangeable sequences with nonatomic marginal distributions as a $\eppf$-scheme.  When \cref{eppfpred} holds, we will say that the $\eppf$-scheme has marginals $\mdrm$.
 
By de~Finetti's theorem, we know there is an a.s.\ tail-measurable random probability measure $\drm$ %
such that
\[
\process Z \Nats \given \drm \distiid \drm.
\]
(We say that $\drm$ is the random measure \defn{directing} the exchangeable sequence $\process Z \Nats$.)
In order to characterize $\drm$ further, let $P_i$ be the \as\ limiting relative frequency of $C_i$, i.e., define
\[
P_i 
\defas \lim_{n\to\infty}  \frac{ N_{in} } n\ \as,\quad i=1,2,\dotsc
\]
Thus $P_1$ is the long run frequency of the first token $\tilde Z_1$.
It is easy to see that the distribution of $\nprocess P \Nats i$ and $\process \nUni \Nats$ depends only on the EPPF $\eppf$ and not on $\mdrm$.  

With probability one, it holds that
\[\label{eq:drmchar}
\drm = \sum_{i=1}^\infty P_i\, \delta_{\tilde Z_i} + \Bigl(1-\sum_{i=1}^\infty P_i \Bigr) \, \mdrm \,.
\]
Note that $\process {\tilde Z} \Nats$ is an \iid-$\mdrm$ sequence, independent of $\process P \Nats$.

Another way of summarizing the combinatorial structure of $\process Z \Nats$ is by the arrival times $\initial {}$ of tokens, i.e., the random function $\initial {} : \Nats \to \Nats$ given by
\[
\initial j \defas \inf \, \theset { i \le j \st Z_i = Z_j },
\] 
i.e., $\initial j$ is the first time the token $Z_j$ appears among $Z_1,\dotsc,Z_j$.  Write $\initial {} \defas ( \initial 1, \initial 2, \dotsc)$. 

\begin{thm}\label{subsampling}
Let $Z$ be a $\pi$-scheme and let $\tau$ be defined as above.   Then there exists an \iid-$\mdrm$ sequence $\process U \Nats$, independent from $\tau$, such that $Z_n = U_{\tau_n}$ \as\ for every $n$.
\end{thm}
\begin{proof}
By an explicit construction and a transfer argument.
\end{proof}

\subsection{Projections}
\label{ssec:proj}

We study a $\eppf$-scheme, where each new token is marked with an \iid\ Bernoulli random variable.  We will be interested in the distribution of the \as\ limiting relative frequency of these marks, as a function of the mean of these Bernoulli marks.  We will be especially interested in the limiting behavior as the mean of these marks converges to zero.  

Let $\process Z \Nats$ be a $\eppf$-scheme with marginals $\mdrm = \Uniform$. 
For every $\qq \in [0,1]$ and $n \in \Nats$, put $\bar Z_{\qq,n} \defas \delta_{Z_n} [0, \qq]$ so that, for every $\qq \in [0,1]$, the process
$(\bar Z_{\qq,n})_{n\in\Nats}$ is an exchangeable sequence in $\theset {0,1}$ directed by the random Bernoulli measure with mean 
\[\label{firstdefnofqp}
Q_\qq \defas \drm [0,\qq],
\]
i.e., conditioned on $Q_\qq$, 
the elements are independent Bernoulli random variables with mean $Q_\qq$.

The process $\nprocess Q {[0,1]} \qq$
can be taken to be the distribution function of the directing random measure $\drm$, 
and so we may choose a version of $Q$ so that $\qq \mapsto Q_\qq$ is monotonically increasing, continuous from the right with left limits, and satisfies
$Q_0 = 0$ and $Q_1 = 1$ surely.
By~\citep[][Prop.~1.4]{PSIP}, for every $\qq \in [0,1]$, we have that $Q_\qq$ is the \as\ limiting frequency of the event $Z_i \le \qq$ as $ i \to \infty$, i.e.,
\[\label{defnofQp}
Q_\qq = \lim_{n\to\infty} \frac{ \card {\theset{ j \le n \st Z_j \le \qq }}  } n\ \as%
\]
It is also clear from~\cref{eq:drmchar} that
\[\label{sumdefnofqp}
Q_\qq = \sum_{i=1}^\infty P_i \delta_{\tilde Z_i} [0,\qq] + \Bigl(1-\sum_{i=1}^\infty P_i \Bigr)\, \qq\, \ \as%
\]
Note that $\process {\tilde Z} \Nats$ is an \iid\ sequence of $\Uniform$ random variables, independent of $\process P \Nats$.
Moreover, the law of $Q$ is fully characterized by $\eppf$, and vice versa.

For $k \le n \in \NNInts$, $\qq \in (0,1]$, and $B\in\BorelSetsInt$,
define
\[\label{bkerneldef}
\bkernel_{n,k}(\qq,B) 
\defas \Pr \theset { Q_\qq \in B \given S_{\qq,n} = k }
=
\frac
{\int_B \pp^k (1-\pp)^{n-k} \,\Pr \theset{ Q_\qq \in \dee \pp }  }
{\int_{[0,1]} \pp^k(1-\pp)^{n-k} \,\Pr \theset{ Q_\qq \in \dee \pp }  } \ ,
\]
where $S_{\qq,n} = \sum_{j=1}^n \bar Z_{\qq,j}$
and
$\kernel (\qq,B) = \bkernel_{0,0}(\qq,B) = \Pr \theset { Q_\qq \in B}$.
Let $j\le m \in \NNInts$.  By Bayes rule, we have
\[\label{eq:bayeschar}
\bkernel_{m+n,j+k}(\qq,B) 
&=
\frac
{\int_B \pp^k (1-\pp)^{n-k} \,\bkernel_{m,j}(\qq,\dee \pp) }
{\int_{[0,1]} \pp^k(1-\pp)^{n-k} \,\bkernel_{m,j}(\qq,\dee \pp) } \ .
\]
\begin{thm}\label{kernellimit}
$\bkernel_{1,1}(q,\argdot) \to \Pr \theset { P_1 \in \argdot }$ weakly as $q\downto 0$.
\end{thm}
\begin{proof} Note that $Z_1 = \tilde Z_1$ a.s.  It follows from \cref{sumdefnofqp},
 that $Q_p$ has the same distribution on $\theset {\tilde Z_1 \le p}$ as $P_1 + \sum_{i=2}^\infty P_i \delta_{\tilde Z_i} [0,q] + \Bigl(1-\sum_{i=1}^\infty P_i \Bigr)\, q$, which is even independent of $\tilde Z_1$.  The latter quantity converges \as\ to $P_1$ as $q\downto 0$, and so its distribution converges to that of $P_1$.
\end{proof}

The preceding result suggests that we can extend $\bkernel_{n,k}$ from $(0,1]$ to $[0,1]$ in such a way that $\bkernel_{n,k}(0,\argdot)$ is defined whenever $\bkernel_{n,k}(q,\argdot)$ is defined for some (and then every) $q \in (0,1)$.

For $k \le n \in \Nats$, define
\[
\ord\kernel_{n,k}(B) \defas \Pr \theset { P_1 \in B \given  \card { ([n] \cap C_1 )} = k }.
\]
Then $\ord\kernel_{1,1} = \sdist$ is the distribution of $P_1$, the \as\ limiting frequency of the first token $\tilde Z_1$.
Bayes rule implies that
\[
\ord\kernel_{n,k}(B) = 
\frac { \int_B p^{k-1} (1-p)^{n-k} \,\sdist(\dee p) }
        { \int_{[0,1]} p^{k-1} (1-p)^{n-k} \,\sdist(\dee p) } \ , \quad B \in \BorelSetsInt.
\]
The next theorem extends $\bkernel$ as $q\downto 0$ and shows that the limiting behavior of $\bkernel$ is determined by $\sdist$:
\begin{lem}\label{limitchar}
Let $k \le n \in \NNInts$. 
As $q \downto 0$, 
\[\label{eq:thelimit}
\bkernel_{n,k}(q, \argdot)
\to 
\begin{cases}
\delta_0 ,
& \text{$k = 0$ or $n=0$,} 
\\
\ord\kernel_{n,k},
& \text{$k\ge 1$,}        
\end{cases}
\qquad \text{weakly.}
\]
\end{lem}
\begin{proof}
From \cref{sumdefnofqp}, it is clear that $Q_q \to 0$ as $q \downto 0$ \as, and so  
$\Pr {\theset {Q_q \in \argdot }} = \bkernel_{0,0}(q,\argdot )$ converges weakly to $\delta_0$. 
To see that this holds when $k=0$, note that $\bkernel_{n,0}(q,A) \to \bkernel_{0,0}(q,A)$ weakly as $q \downto 0$
because $\Pr \theset {S_{q,n} = 0} \to 1$ as $q \downto 0$.

Now assume $k \ge 1$.  The result follows from \cref{kernellimit}, the identity \cref{eq:bayeschar} with $m=j=1$, and the boundedness and continuity of the map $p \mapsto p^{k-1} (1-p)^n$ on $[0,1]$.
\end{proof}

\cref{limitchar} implies that we may define
\[
\bkernel_{n,k}(0,\argdot) = \lim_{q \downto 0} \bkernel_{n,k}(q,\argdot)
\]
where the limit is understood in terms of weak convergence (not setwise).
Recall that $\PFF {n}(p) = (1-p)^n$.
We can give the following characterization of $\PFF {n}\sdist$ in terms of $\process P \Nats$ and $\process \nUni \Nats$:
\begin{lem}\label{charofmu}
For every $n \in \Nats$ and $B \in \BorelSetsInt$, we have
\[
(\PFF {n-1 }\sdist) (B) = \Pr \theset{ P_{\nUni_n} \in B \wedge \nUni_n > \nUni_{n-1} }.
\]
\end{lem}
\begin{proof}
We have 
$\Pr^{P_1} \oF{ C_1 \cap [n] = 1 \wedge P_1 \in B } = (1-P_1)^{n-1} \Ind_{B} (P_1)$.  
The expectation of the latter is $\PFF {n-1 }\sdist B$.
But then, by exchangeability,
$
\Pr \theset { C_1 \cap [n] = 1 \wedge P_1 \in B }
=
\Pr \theset { C_{\nUni_n} \cap [n] = 1 \wedge P_{\nUni_n} \in B }
$ and $\theset {C_{\nUni_n} \cap [n] = 1} = \theset{\nUni_n > \nUni_{n-1}}$.
\end{proof}

Define $\Delta_n \defas \Pr \theset { \nUni_n > \nUni_{n-1} }$.  From the proceeding result, we have $\Delta_n = \sdist \PFF {n-1}$.
\section{Exchangeable sequences of Bernoulli processes}

For a measure $\nR$ on $\bspace$,
let $\countproc {\nR}$ be the simple measure 
on $\bspace \times (0,\infty)$
given by $\nR = \sum_{s \in \bspace} \delta_{(s,\nR\theset s)}$.  It can be shown that the map $\nR \mapsto \countproc {\nR}$ is measurable.
Let $\charber x p  \defas 1-p+p\,e^{-x }$ be the moment generating function of a Bernoulli random variable with mean $p$ evaluated at $-x$, and, given a measurable function $f$ on $\bspace$, write $\charbersym f$ for the map $(s,p) \mapsto \charber {f(s)} p$.  

\begin{thm}\label{exchjointlaw}
Let $g,f_1,\dotsc,f_n \ge 0$ be measurable, and put $g'(s,p) = p\, g(s)$. 
Let $\rH$ be a completely random hazard measure with nonrandom nonatomic component $\nrm$, fixed atoms $\Atoms$, and \Levy\ measure $\eta$ on $\bspace \times (0,1]$.
Conditioned on $\rH$, let $X_1,\dotsc,X_n$ be independent Bernoulli processes with mean $\rH$.  Then
\[
\log \EE \exp \Oof{ - \rH g - \sum_{j=1}^n X_j f_j } 
&= - \nrm (g +  \sum_{j=1}^{n} (1- e^{-f_j})) - \eta (1- e^{-g'} \prod_{j=1}^n \charbersym {f_j} ) 
\\&\quad \ + \sum_{s \in \Atoms} \log \Ooof{ (\delta_s \otimes \Pr \oF { \rH \theset s }) (e^{-g'} \textstyle\prod_{j=1}^n \charbersym {f_j}) }.
\]
\end{thm}
\begin{proof}
We have
\[
\EE^{\rH }[ \exp \Oof {- X_j f_j } ] = \exp \Ooof { - \nrm(1 - e^{-f_j}) + 
                          \countproc {\rH} \log \charbersym {f_j}  }.
\]
It follows from the chain rule of conditional expectation and then
\cref{thm:crmrep}
that
\begin{multline}\label{maineq123}
\EE \exp \Oof{ - \rH g - \textstyle\sum_{j=1}^n X_j f_j } 
= \EE \oooF { \exp \Oof { - \rH g } \prod_{j=1}^n \EE^{\rH}[ \exp \Oof {- X_j f_j } ] }
\\
= \EE[ \exp \Oof { - \nrm (g +  \textstyle\sum_{j=1}^{n} (1- e^{-f_j})) - \countproc {\rH} (g' - \log \prod_{j=1}^n \charbersym {f_j} ) } ].
\end{multline}
Let $\rH_\Atoms \defas \rH(\argdot \cap \Atoms)$.
Then by Campbell's theorem,
\[\label{capm234}
\EE[ \exp \Oof { - \countproc {\rH-\rH_\Atoms} (g' - \log \textstyle\prod_{j=1}^n \charbersym {f_j} ) } ] 
= \exp \Ooof { - \eta (1- e^{-g'} \textstyle\prod_{j=1}^n \charbersym {f_j} ) }.
\]
On the other hand, by complete randomness, $\rH_\Atoms$ and $\rH-\rH_\Atoms$ are independent and
\[
\EE[ \exp \Oof { - \countproc {\rH_\Atoms} (g' - \log \textstyle\prod_{j=1}^n \charbersym {f_j} ) } ]
= \prod_{s \in \Atoms} \Pr \oF { \rH \theset s } (e^{-g'} \textstyle\prod_{j=1}^n \charbersym {f_j}),
\]
completing the proof.
\end{proof}

We may write $\eta = \lbm \otimes \lkernel$ for some measure $\lbm$ on $\bspace$ and kernel $\lkernel$ from $\bspace$ to $(0,1]$.
In the case that $\nrm$ is absolutely continuous with respect to $\lbm$, we can give a unified characterization of the ordinary components of $X_1,\dotsc,X_n$.

\begin{thm}\label{alternativeexchjointlaw}
Let $\lbm \otimes \lkernel$ be a disintegration of $\eta$,
assume that $\nrm \ll \lbm$ and let $\Delta$ satisfy $\nrm = \Delta \lbm$.
Then
\[
&-\textstyle\sum_{j=1}^{n} \nrm (1- e^{-f_j}) 
      - \eta ( 1 - \prod_{{j}=1}^{n} \charbersym {f_j} )
\\
&=
n(\lbm \otimes (\PF 1 1 \lkernel + \Delta \delta_0)    
    \Bin^{(n-1)}) (s,j \mapsto 
      \frac 1 {j+1} \sum_{\substack{J \subseteq [n]\\\card {J} = j+1}} {n \choose j+1}\inv [1 - e^{-f_J(s)} ] )
\]
\end{thm}

The partial sums $S_n \defas X_1+ \dotsm + X_n$ are key quantities when studying the conditional distribution of $\rH$.  The next result characterizes the joint law of $S_n$ and $H$.  Write $\ord \rH \defas \rH( \argdot \setminus \Atoms)$ and $\ord S_n \defas S_n (\argdot \setminus \Atoms)$,
and let $\mathrm{Bin}(n,p)$ denote the Binomial distribution on $\theset{0,1,\dotsc,n}$ with mean $np$ and variance $np(1-p)$.
For every $r \le n \in \NNInts$, 
define $\PF {n} {r} : [0,1] \to [0,1]$ by
$\PF {n} {r} (p) \defas p^{r} (1-p)^{n-r}.$

\begin{cor}\label{exchjointlawofS}
Let $f,g \ge 0$ be measurable, 
let $f''(s,p,k)=f'(s,k)=k\,f(s)$ and $g''(s,p,k)=g'(s,p) = p\, g(s)$,
and let $\Bin^{(n)}(p,\argdot) = \Bin(n,p)$.
Then
\[\label{jointlaweqnofS}
\log \EE e^{ - \ord\rH g - \ord S_n f } 
&=  - \nrm (g + n\,(1- e^{-f})) - \eta (1- e^{-g'} \charbersym{f}^n ) 
\\&=  - \nrm (g + n\,(1- e^{-f})) - (\lbm \otimes (\lkernel \otimes \Bin^{(n)})) (1- e^{-g'' - f''} ) ,
\]
where
$\lbm \otimes \lkernel$ is a disintegration of $\eta$.
In particular, for $h\ge 0$ measurable, 
\[
\EE e^{ - \countproc{\ord S_n} h } 
= \exp \Oof { - (n \nrm \otimes \delta_1
                    + \lbm \otimes \lkernel \Bin^{(n)}_{[n]} )
                     (1- e^{-h} ) },
\]
where $\Bin^{(n)}_A(p,\argdot)$ is the restriction of the measure $\Bin{(n,p)}$ to the set $A$.
\end{cor}
\begin{proof}
The first equality follows from \cref{exchjointlaw}, taking $f_j=f$.
To establish the second equality, note that
\begin{multline}
\qquad \charbersym{f}^n(s,p) 
= (1-p+p\,e^{-f(s)})^n 
\\= \sum_{k=0}^n {n \choose k} p^k (1-p)^{n-k} e^{-k f(s)}
= \mathrm{Bin}(n,p) (k \mapsto e^{- k f(s)}). \qquad
\end{multline}
The second claim follows from \cref{thm:crmrep} after taking $g=0$, and noting that $\countproc {\ord S_n}$ is a Poisson process on $\bspace \times \Nats$.
\end{proof}

For a kernel $\lkernel$ from $S$ to $T$ let $\drag \lkernel$ be the kernel from $S$ to $S \times T$ given by $\drag \lkernel_s = \delta_s \otimes \lkernel_s$.
For a finite measure $\tau$ on a space $S$, let $\normalize {\tau}$ be the probability measure $\tau / (\tau S )$.

\newcommand{\postkappa}{\wp}
\newcommand{\df}{\drag f}
\newcommand{\dg}{\drag g}
\begin{thm}\label{exchpostlaw}
Let $g \ge 0$ be measurable,
let $\ord \rH$ and $\ord S_n$ be as above,
let $\lbm \otimes \lkernel$ be a disintegration of $\eta$,
assume that $\nrm \ll \lbm$, and let $\Delta$ satisfy $\nrm = \Delta \lbm$.
Then \as
\[
\log \EE^{\ord S_n} \oooF{ e^{ - \ord \rH g } } 
= %
      -\nrm g 
      - (\lbm \otimes \PF n 0 \lkernel) (1-e^{- g'}) 
      + \countproc{ \ord S_n } \log { \drag\postkappa^{(n)} e^{-\dg} } 
   ,
\]
where $\dg(s,k,p)\defas g'(s,p) \defas p\,g(s)$
and 
\[
\postkappa^{(n)}_{s,k} \defas
     \overline { 
       \Delta(s) \PF {n-1} {k-1} \delta_0 
       +
       \PF n k \lkernel_s  
       }
       .
\]
\end{thm}

\newcommand{\cmeasure}{\mathfrak c}
For $n \ge k \ge 1$, note that $\PF {n-1} {k-1} \delta_0 $ is $\delta_0$ when $k=1$ and is otherwise the null measure.
The following result is the key identity.
\begin{lem}\label{exchpostlawlemma}
Define $\pi(s,p,k) = (s,k,p)$.
Then $\postkappa \defas \postkappa^{(n)}$ satisfies 
\[
(\Delta \lbm \otimes n\delta_1 + \lbm \otimes \lkernel \Bin^{(n)}_{[n]} ) \otimes \postkappa
=(\Delta\lbm \otimes \delta_0 \otimes n\delta_1 + \lbm \otimes \lkernel \otimes \Bin^{(n)}_{[n]} ) \circ \pi\inv.
\]      
\end{lem}
\begin{proof}
Let $h \ge 0$ be measurable, and define
\[
 h'(s,p,k) \defas \frac { h(s,k,p) } { g'(s,p,k') + f'(s,p,k)}
\]
where $g'(s,p,k) \defas \lkernel_s \PF n k$ 
and 
\[
 f'(s,p,k) \defas \Delta(s)\,\PF {n-1} {k-1} (0) = \Delta(s)\,\delta_1 \{k\} .
\]
Let $\cmeasure_n$ denote counting measure on $[n]$, and let $\cmeasure^b_n \defas {n \choose \argdot} \cmeasure_n$.  Noting that  
$\Bin^{(n)}_{[n]}(p,\argdot) = \PF n \argdot (p) \cmeasure^b_n$ 
and $\Bin^{(n)}_{[1]}(p,\argdot) =  \PF n 1 (p) \, n \delta_1$, 
it is straightforward to verify that
\[
(\Delta \lbm \otimes n\delta_1 \otimes \postkappa ) h
=
 (\Delta\lbm \otimes \delta_0 \otimes n \delta_1) f' h'
+  (\lbm \otimes \lkernel \otimes \Bin^{(n)}_{[n]}) f'h'
\]
and
\[
(\lbm \otimes \lkernel \Bin^{(n)}_{[n]} \otimes \postkappa ) h
= (\Delta\lbm \otimes \delta_0 \otimes n \delta_1) g'h'
          + (\lbm \otimes \lkernel \otimes \Bin^{(n)}_{[n]}) g'h'
 .
\]
Summing and using the identity $(f' + g')h' = h \circ \pi$ completes the proof.
\end{proof}

\begin{proof}[Proof of \cref{exchpostlaw}]
Let $f\ge 0$ be measurable.
By the chain rule of conditional expectation,
\[
\EE e^{- \ord\rH g - \ord S_n f }
= \EE \ooF{ e^{-\ord S_n f } \EE^{\ord S_n} \ooF{ e^{-\ord \rH g} } }
\]
and so it suffices to show that \cref{jointlaweqnofS} in \cref{exchjointlawofS} is equal to  
\[
\exp \Oof{ 
      -\nrm g 
      - (\lbm \otimes \PF n 0 \lkernel) (1-e^{- g'}) }
 \, \EE \exp \Oof{ \countproc{ \ord S_n } \log {\drag\postkappa^{(n)} e^{-\df-\dg} } },
\label{kthetnhesuhe}
\]
where $\df(s,k,p) = k\, f(s)$.
By \cref{exchjointlawofS}, the identity $\nrm = \Delta \lbm$, and the fact that $\postkappa^{(n)}$ is a probability kernel,
\[
\begin{split}\label{asoehusaeo}
&\log \EE \exp \Oof{ \countproc{ \ord S_n } \log {\drag\postkappa^{(n)} e^{-\df-\dg} } }
\\&\qquad=
 - (n \Delta\lbm \otimes \delta_1 + \lbm \otimes \lkernel \Bin^{(n)}_{[n]} )
      \drag\postkappa^{(n)} (1-e^{-\df-\dg} ) 
      .
\end{split}
\]
By \cref{exchpostlawlemma} and the identity
\[
\lkernel \otimes \Bin^{(n)}
 = 
      \lkernel \otimes \Bin^{(n)}_{[n]}
         +
       \PF n 0 \lkernel \otimes \delta_0 ,
\]                    
we can rewrite the right hand side of \cref{asoehusaeo} as
\[
\begin{split}
          -n\nrm(1-e^{-f}) 
                    &- (\lbm \otimes (\lkernel \otimes \Bin^{(n)})) (1-e^{-f''-g''}) 
                    \\& + (\lbm \otimes \PF n 0 \lkernel) (1-e^{- g'})
       ,
\end{split}
\]
where $f'' = \df \circ \pi$ and $g'' = \dg \circ \pi$.
Substituting back into \cref{kthetnhesuhe}, we find agreement with \cref{jointlaweqnofS}, completing the proof.
\end{proof}

In the next section, we introduce a generalization of the one-parameter scheme.
\section{The continuum-of-urns scheme}
\label{sec:cup}

In this section, we define a generalization of the one-parameter scheme;
show that it is an exchangeable sequence of Bernoulli processes;
and then characterize the directing random \hazard\ measure.  
As a result we define a family of completely random \hazard\ measures generalizing beta processes.  
\NA{These random measures can be similarly organized into hierarchies and 
admit a straightforward urn scheme characterizing directed \iid\ sequences of Bernoulli processes.}

\subsection{Construction}

\newcommand{\EQ}{\EE_{\nY}}
Let $\nY \defas \process \nY \Nats$ be a nonrandom sequence of simple measures on $\borelspace$ 
concentrated on a nonrandom locally-finite countable set $\Atoms$,
and let $\EPPFfam$ be a \defn{(measurable) family of EPPFs}, i.e., one such that the map $s \mapsto \pi_s(n)$ is measurable for every $n \in \Nats^*$.
For every $s \in \Atoms$, let $Z^s \defas (Z^s_{1},Z^s_{2},\dotsc)$ be an independent $\eppf_s$-scheme, 
and let $\ginitial {} s$ be the arrival times of the tokens in $Z^s$.
Consider the sequence $\nX \defas \process \nX \Nats$ of simple point process, concentrated on $\Atoms$,
where, for every $n \in \Nats$ and $s \in \Atoms$,
\[\label{simpledef}
 \nX_n \theset s = \nY_{\ginitial n s} \theset s. 
\]
Equivalently,
$\nX_1 \defas \nY_1$
and for every $n \in \Nats$ and $s \in \Atoms$, 
\[\label{maindef}
\nX_{n+1} \theset s \defas 
\begin{cases}
\nX_{j} \theset s, & \text{ if $Z^s_{n+1} = Z^s_j$, where $j \le n$,} \\
\nY_{n+1} \theset s, & \text{ otherwise.} 
\end{cases}
\]
Let $\cQ{\nY}$ be the law of $\nX$, which is clearly a measurable function of $\nY$.

\begin{definition}[\COUS]
Let $\BM$ be a \hazard\ measure on $\borelspace$ and let $\process X \Nats$ be a sequence of random measures on $\borelspace$.
We call $\process X \Nats$ a \defn{\COUS\ induced by the (measurable) EPPF family $\EPPFfam$ and \hazard\ measure $\BM$} when there exists an \iid\ sequence $\process Y \Nats$ of Bernoulli processes with \hazard\ measure $\BM$ such that 
\[\label{mainconddef}
\Pr[\process X \Nats \given \process Y \Nats] = \cQ{\process Y \Nats} \ \as
\]  
We may also say that $\process X \Nats$ is a \COUS\ \defn{induced by $\EPPFfam$ and $\process Y \Nats$}.
A \COUS\ induced by a EPPF family $\EPPFfam$ is said to be \defn{homogeneous} if, for some EPPF $\eppf$ and every $s \in \bspace$, we have $\eppf_s = \eppf$, and \defn{nonhomogeneous} otherwise.
\end{definition}

\begin{remark}[relationship with one-parameter scheme] \label{rem:CFopus}
Let $\conc > 0$ and consider the EPPF $\eppf$ given by  
\[\label{opeppf}
\eppf (n_1,\dotsc,n_k) \defas \frac { \conc^{k-1} \prod_{i=1}^k (n_i - 1)! } { [1+ \conc]_{n - 1} },
\]
where $n = \sum_{i=1}^k n_i$ and $[x]_m \defas \prod_{j=1}^m (x+j-1)$.  
This EPPF is that associated with the exchangeable sequence characterized by \cref{bmus}, i.e., a Blackwell-MacQueen urn scheme.
It is therefore immediate that the above definition of a \COUS\ specializes to that of the one-parameter scheme with concentration parameter $\conc$.
Thus, in this special case, $\process X \Nats$ is an exchangeable sequence of Bernoulli processes directed by a beta process.
We can generalize the one-parameter scheme by fixing a (measurable) concentration \emph{function} $\conc : \bspace \to \NNReals$ and constructing a \COUS\ induced by the measurable family $\EPPFfam$ where $\eppf_s$ is given by \cref{opeppf} for $\conc=\conc(s)$.
In this case, the sequence remains exchangeable and directed by a nonhomogeneous beta process.  (\citep{Thibaux2007} discuss this particular nonhomogeneous case.)  The next few results show that a general \COUS\ is also exchangeable and directed by a completely random \hazard\ measure.
\end{remark}

\begin{remark}[simulability]
Note that in a computer simulation of the processes $\nX_1,\dotsc, \nX_n$, one need only generate $Z^s_1,\dotsc,Z^s_n$ for each $s \in \nY_j$ and $j \le n$.
\end{remark}

For the remainder of the section, let $\process X \Nats$ be a \COUS\  induced by the measurable family $\EPPFfam$ 
and an \iid\ sequence $\process Y \Nats$ of Bernoulli processes with \hazard\ measure $\BM$.  

\subsection{\Triangular\ construction}

In order to characterize the law of $\process X \Nats$, it will be useful to return to its construction and produce a richer process from which we can derive $\process X \Nats$.
For every $s \in \Atoms$, let $\ftt s {} \defas (\ftt s 1, \ftt s 2, \dotsc)$ be the first arrival times of the unique tokens in $Z^s$, and recall that, on the event that there are only $j$ unique tokens, $\ftt s k = \infty$ \as\ for every $k > j$.
Let $\indexmk {\nX} {n} {m} {k} $, for ${k} \le {m} \le {n} \in \Nats$, be the \triangular\ array of simple point processes, concentrated on $\Atoms$, 
such that, for every $s \in \Atoms$, 
\[
\indexmk {\nX} {n} {m} {k} \{s\}
 \defas 
 1(\ftt s {k} = {m})\, 1(\tau^s_{n} = m) \, \nY_{m}\{s\} .
\]
It is easy to verify that,
for every ${n} \in \Nats$, 
\[
\nX_{n} = \sum_{{k} \le {m} \le {n}} \indexmk {\nX} {n} {m} {k} \ \as
\]
Writing $\cQ{\nY}^*$ for the law of the \triangular\ array of simple point processes,
a transfer argument 
implies that there exists a \triangular\ array of simple point processes $\indexmk {X} {n} {m} {k}$
such that
\[\label{thithutuihui}
\Pr [ \indexmk {X} {n} {m} {k} : {k} \le {m} \le {n} \in \Nats \given Y ] = \cQ{Y}^* \ \as
\]
and therefore, for every $n \in \Nats$, 
\[
X_{n} = \sum_{{k} \le {m} \le {n}} \indexmk {X} {n} {m} {k} \ \as
\]

\subsection{The law via characteristic functionals}

We begin by characterizing the probability kernel $\nY \mapsto \cQ {\nY}^*$.
For every 
$n \in \Nats$
and family $f=\gprocess {\indexmk f {j} {m} {k} } {{k} \le {m} \le {j} \le {n}}$ of nonnegative measurable functions on $\bspace$,
it follows from the independence of the sequences $Z^s$ that
\[\label{theutheuthiu}
\log \EE \ooof { \exp \Ooof { - \textstyle \sum_{{k} \le {m} \le {j} \le {n}}
            \indexmk {\nX} {j} {m} {k} \indexmk f {j} {m} {k} }}
 = 
\sum_{s \in \Atoms}
\log \lambda_{f} (s, \nY_1\{s\},\dotsc, \nY_n\{s\} )
\]            
where
\[
\lambda_{f} (s,y_1,\dotsc,y_n)
=
  \EE_{\eppf_s} \ooof { \exp \Ooof { -
     \textstyle \sum_{{k} \le {m} \le {j }\le {n}}
               1(M_{k} = {m})\,
               1(\tau_{j} = m)\,
               y_m \,
                \indexmk f {j} {m} {k} (s) }}.
\]
Introduce on $\bspace \times ( \theset {0,1}^n \setminus \theset 0^n) $
the measure
\[
U_n \defas \sum_s \delta_{(s,(Y_1 \theset s,\dotsc,Y_n \theset s))}.
\]
It follows from \cref{theutheuthiu,thithutuihui} that \as\
\[\label{funcY2}
\log \EE^{Y} \oooF { \exp \Ooof { - \textstyle \sum_{{k} \le {m} \le {j} \le {n}}
            \indexmk {\nX} {j} {m} {k} \indexmk f {j} {m} {k} } 
            }
 = U_{n} \log \lambda_{f} .
\]            
\begin{prop}\label{altlaplacefuncs}
Let $R_n \defas Y_1 + \dotsm + Y_n$
and define 
\[\label{blockver} 
\Lambda_{f}(s,r) \defas 
  \sum_{y : \sum_j y_j = r} {n \choose r}\inv 
        \lambda_{f}(s,y_1,\dotsc,y_n)
   . 
\]
Then \as
\[
\EE^{R_n} \ooF { \exp \Oof{-\textstyle \sum_{{k} \le {m} \le {j} \le {n}}  \indexmk {X} {j} {m} {k} \indexmk f {j} {m} {k}  }} = \exp \Oof{ \countproc {R_n} \log \Lambda_{f} }. 
\]
\end{prop}
\begin{proof}
Let $g,g_1,\dotsc,g_n \ge 0$ be measurable functions
and define $g'(s,k) \defas k\,g(s)$.
Then 
\[\label{Rlaplacefunc}
\log \EE \exp \of { - R_{n} g }
= - {n} \NABM (1 - e^{-g}) 
  + \countproc {\BM} \log \drag \Bin^{({n})} e^{-g'}.
\]
It is straightforward to verify that \as
\[\begin{split}
&\log \EE^{R_{n}} \oooF {
    \exp \Ooof { - \textstyle \sum_{{j} \le {n}} Y_j g_j } 
    }
 \\&\qquad= 
       \countproc {R_n} \log
       \ooof{ 
          s,r \mapsto \textstyle \sum_{y : \sum_j y_j = r} {n \choose r}\inv 
                                \exp \Ooof { - \textstyle \sum_{{j} \le {n}} y_j g_j(s)} 
         }
         ,
\end{split}         
\]
from which we can infer that \as 
\[\label{estbhithiusu}
&\log \EE^{R_{n}} \oooF {
    \exp \Ooof { - U_n h } 
    }
 \\&\qquad= 
       \countproc {R_n} \log
       \ooof{ 
          s,r \mapsto \textstyle \sum_{y : \sum_j y_j = r} {n \choose r}\inv 
                                \exp \Ooof { - h(s,y_1,\dotsc,y_n) }
         }
         .
\]
The proof then follows from \cref{funcY2,estbhithiusu}.
\end{proof}

The following result characterizes the law of the \triangular\ array, and highlights the distinct roles played by the atoms $\FixedAtoms$
and the nonatomic part $\NABM$ of the \hazard\ measure $\BM$.

\begin{prop}\label{altlaplacefuncofjoint}
Let $n \in \Nats$, 
let $\mathrm{Bin}_{(n,p)}$ denote the distribution of the sum of $n$ independent Bernoulli random variables, each with mean $p \in (0,1]$,
and let $f=\gprocess {\indexmk f {n} {m} {k} } {{k} \le {m} \le {n} \in \Nats}$ be a family of nonnegative measurable functions on $\bspace$. 
Then
\[
\log \EE \exp \Oof {-\textstyle \sum_{{k} \le {m} \le {j} \le {n}}
            \indexmk {X} {j} {m} {k} f_{j,{m},{k}} }
= - n \NABM (1 - \Lambda_f (\argdot, 1))
     + \countproc {\BM} \log \drag \Bin^{(n)} \Lambda_f
\]
\end{prop}
\begin{proof}
Follows from \cref{altlaplacefuncs,Rlaplacefunc}.
\end{proof}

We now characterize the law of $\process X \Nats$.

\begin{prop}\label{laplacefuncofjoint}
Let $f =\nprocess f {[n]} j$ be a family of nonnegative measurable functions on $\bspace$
and let
\[
\Lambda_{f}^* (s,r)
= \sum_{y : \sum_j y_j = r} {n \choose r}\inv 
  \EE_{\eppf_s} \ooof { \exp \Ooof { -
     \textstyle 
            \sum_{{j} \le {n}}
                y_{\tau_{j}}
                f_j (s)
               }}
        .
\]
Then
\[
\log \EE \exp \Oof {-\textstyle\sum_{j \le n} X_j f_j }
= - n \NABM (1 - \Lambda_f^* (\argdot, 1))
     + \countproc {\BM} \log \drag \Bin^{(n)} \Lambda_f^*.
\]
\end{prop}
\begin{proof}
Follows from \cref{altlaplacefuncofjoint} for $\indexmk f {n} {m} {k} = \index f_n$.
\end{proof}

While \cref{laplacefuncofjoint} characterizes the law of $X$, it is relatively opaque.  Consider the fixed component:  We have
\[
\Bin^{(n)} \Lambda_f^*(s,\argdot)
&= (q \mapsto \EE_{\eppf_s}
      \ooof { \exp \Ooof { -
     \textstyle 
            \sum_{{j} \le {n}}
               1(U_{\tau_j} \le q)
                f_j (s)
               }}
\\&= (q \mapsto \EE_{\eppf_s}
      \ooof { \exp \Ooof { -
     \textstyle 
            \sum_{{j} \le {n}}
                Z^q_j
                f_j (s)
               }}
\\&= \thekernel_s \prod_{{j} \le {n}} \charbersym {f_j} (s, \argdot),
     \label{thiuthiutihui}
\]
where $U_1,U_2,\dotsc$ are independent uniformly distributed random variables independent of $\tau_1,\tau_2,\dotsc$ satisfying $Z_n = U_{\tau_n}$ a.s.

\newcommand{\unin}{\ceiling{nU}}
Now consider the ordinary component: Let $U$ be a uniformly distributed random variable, independent from $\process \tau \Nats$. We have
\[
n( 1 - \Lambda_f^* (s, 1))
&=  
   \sum_{y : \sum_j y_j = 1} 
  \EE_{\eppf_s} \ooof { 1- 
   \exp \Ooof { -
     \textstyle 
            \sum_{{j} \le {n}}
                y_{\tau_{j}}
                f_j (s) 
               } }
\\
&= 
  \EE_{\eppf_s} \of {
   1- \exp \Ooof { -
     \textstyle 
            \sum_{{j} \le {n}}
                1(\tau_{j} = {\unin})\,
                f_j (s)
               }}
\\
&= 
  \EE_{\eppf_s} \of {
  \frac 1 {N_{K_{\tau_{\unin}},n}}
  (1 - 
   \exp \Ooof { -
     \textstyle 
            \sum_{{j} \le {n}}
                1(\tau_{j} = \tau_{\unin})\,
                f_j (s)
               }
          )
      }
      .
\]
By exchangeability, for all ${m} \le {n}$, 
\[
N_{K_{\tau_{\unin}},n} \equaldist N_{1,n}
\]
and, conditioned on $N_{K_{\tau_{\unin}},n} = k$,
the vector $(1(\tau_1 = \tau_{\unin}),\dotsc,1(\tau_n = \tau_{\unin}))$ is uniformly distributed among those vectors with $k$ ones and $n-k$ zeros.
Finally, note that, conditioned on $P_1$, it is the case that $N_{1,n}-1$ is binomially distributed with mean $(n-1)P_1$ and variance $(n-1)P_1(1-P_1)$.
It follows that
\[
n( 1 - \Lambda_f^* (s, 1))
&= 
  \Pr_{\eppf_s} \oF { N_{1,n} }
  (k \mapsto
      \frac 1 k
      \sum_{y : \sum_j y_j = k}
        (1 - \exp \Ooof { -
	        \textstyle 
        		    \sum_{{j} \le {n}}
        		        	y_j \, f_j (s)
               }
          )
      )
\\&= 
  \Pr_{\eppf_s} \oF { P_{1} }
  \Bin^{(n-1)}
  (k \mapsto
      \frac 1 {k+1}
      \sum_{y : \sum_j y_j = k+1}
        (1 - \exp \Ooof { -
	        \textstyle 
        		    \sum_{{j} \le {n}}
        		        	y_j \, f_j (s)
               }
          )
      )  \label{aosnhtintihusu}
      .
\]

Let $\Delta(s) = \Pr_{\eppf_s} \{ P_1 = 0\}$ for $s \in \bspace$.

\begin{thm}\label{maindefinetti}
Let $\cF_n = \sigma(X_1,\dotsc,X_n)$ and 
define the partial sums $S_n \defas X_1 + \dotsm + X_n$, for $n \in \Nats$.
There is an \as\ unique, completely random \hazard\ measure $\rH$ such that
\[\label{definettiLLN}
\rH(A) = \lim_{n\to\infty} {\textstyle \frac 1 n} S_n (A) \ \as, \text{ for every $A \in \bsa$,}
\]
and
\[\label{bepdisint}
\process X \Nats \given \rH \dist \BePLAW(\rH).
\]
In particular, $\process X \Nats$ is an exchangeable sequence of Bernoulli processes.
For every $s \in \bspace$ and $A \in \bsa$, we have $\Pr \{\rH \{s\} \in A \} = \thekernel(\NABM\{s\}, A)$.
Moreover, conditional on $\cF_n$, the law of $\ord \rH = \rH ( \argdot \setminus \FixedAtoms)$ 
is given by \as\
\[
\log \EE^{\cF_n} \oooF{ e^{ - \ord \rH g } } 
= %
      - \Delta \NABM g 
      - (\NABM \otimes \PF {n-1} {-1} \sdist) (1-e^{- g'}) 
      + \countproc{ \ord S_n } \log { \drag\postkappa^{(n)} e^{-\dg} } 
   ,
\]
where $\dg(s,k,p)\defas g'(s,p) \defas p\,g(s)$
and 
\[
\postkappa^{(n)}_{s,k} \defas
     \overline {
       \PF {n-1} {k-1} \sdist
       }
       .
\]
\end{thm}
\begin{proof}
Follows from comparison of \cref{thiuthiutihui,aosnhtintihusu} with \cref{exchjointlaw,alternativeexchjointlaw} after taking
$\nrm \defas \Delta \NABM$,
$\lbm \defas \NABM$, and 
$\lkernel(s,\dee p) \defas p\inv \sdist(s,\dee p)$.
\end{proof}

\begin{remark}
\cref{intromain,genbeta} follow immediately as corollaries.
\end{remark}

\begin{definition}
We will say $\rH$ is the \defn{random \hazard\ measure directing} $\process X \Nats$.
\end{definition}

\begin{remark}
Recall the family of EPPFs defined in \cref{rem:CFopus}, corresponding to a continuum of Blackwell-MacQueen urn schemes.
In this case we know that the correspondence with the one-parameter scheme implies that the directing random \hazard\ measure is a beta process.  
\end{remark}

\begin{remark}
Conditioned on $\cF_n$, the directing random \hazard\ measure $\rH$ is completely random with
\nonrandomcomp\ component $\Delta \NABM$,
fixed atoms $\FixedAtoms \cup \supp {S_n}$,
and ordinary component with \Levy\ measure
\[
p\inv (1-p)^n\,\sdist(s,\dee p) \, \NABM(\dee s)
\]
For $s \in \Atoms$,
 the distribution of $\rH\{s\}$ given $\cF_n$ is 
\[
\normalize {\PF {n} {S_n\{s\}} \thekernel (\BM \{s\}, \argdot)  }.
\]
Informally speaking, for $s \in \supp S_n \setminus \FixedAtoms$,
the distribution of $\BM \{s\}$ given $\cF_n$ is
\[
\postkappa^{(n)}(s,S_n\{s\}).
\]
\end{remark}

\subsection{Alternative characterizations of $\rH$}  

For ${k} \le {m} \le {n} \in \Nats$, define the partial sums
\[
\indexm X {n} {m} &\defas \sum_{{k'}=1}^{m} \indexmk X {n} {m} {k'},
&\indexk X {n} {k} &\defas \sum_{{m'}={k}}^{n} \indexmk X {n} {m'} {k}, 
\\
\indexm S {n} {m} 
 &\defas \sum_{{j}={m}}^{n} \indexm X {j} {m},
&\indexk S {n} {k} 
 &\defas \sum_{{j}={k}}^{n}  \indexk X {j} {k} .
\]
We will give two complementary characterizations of the directing random \hazard\ measure $\rH$ on $\bspace \setminus \Atoms$ using the identities
\[\label{mainidentity}
S_{n} = \sum_{{m}=1}^{n} \indexm S {n} {m} = \sum_{{k}=1}^{n} \indexk S {n} {k} \, .
\]
We will recover two classes of representations that have been described in the special case where $\eppf$ corresponds to a one- and two-parameter Chinese restaurant process.  The first class would be well described as \emph{size-biased} and corresponds with the first equality in \cref{mainidentity}.  Such a representation was given by Thibaux and Jordan \citep{Thibaux2007} in the one-parameter case and by \TG~\citep{TehGor2009a} in the two-parameter case.

The second equality in \cref{mainidentity} corresponds with the second class of representations.  These are the so-called \emph{stick-breaking} representations, although this name would have perhaps been best reserved for the Ferguson--Klass-type construction given by Teh, \Gorur\, and Ghahramani \citep{TehGorGha2007} of the $c=1$ instance of the one-parameter case.   For the second class, we recover the stick-breaking construction of Paisley, Zaas, Woods, Ginsburg, and Carin~\citep{PZWGC2010}, which was later revisited by Paisley, Blei, and Jordan~\citep{PBJ2012}
using the calculus of completely random measures.  When $\eppf$ corresponds with a two-parameter Chinese restaurant process, we recover the stick-breaking construction described by Broderick, Jordan, and Pitman~\citep{BJP12}.  
In a sense, both classes of alternative representations follow in trivial fashions from identities that we note here for the first time in the stick-breaking class.  However, the development below describes much more of the probabilistic structure.

As our focus is on the nonatomic part of $\rH$, 
define $\indexmk {\ord X} {n} {m} {k} \defas \indexmk X {n} {m} {k} (\argdot \setminus \FixedAtoms)$ 
and similarly define ${\ord X}_{n}$, $\indexm {\ord X} {n} {m}$, $\indexk {\ord X} {n} {k}$, $\indexm {\ord S} {n} {m}$, $\indexk {\ord S} {n} {k}$, and ${\ord S}_{n}$.
We begin by showing that the columns of $\indexm {\ord X} {n} {m}$ and $\indexk {\ord X} {n} {k}$ are independent:

\begin{thm}\label{theindplawethiu}
The columns
\[
  \ooof { \indexmk {\ord X} {n} {m} {k} \st {k} \le {m} \text{ and } {n} \ge {m} }, \qquad \text{ $m \in \Nats$, }
\]
are independent.  The same holds for the columns
\[
  \ooof { \indexmk {\ord X} {n} {m} {k} \st {m} \ge {k} \text{ and } {n} \ge {m} }, \qquad \text{ $k \in \Nats$. }
\]
\end{thm}
\begin{proof}
Let $f=\gprocess { \indexmk f {n} {m} {k} } {{k} \le {m} \le {n} \in \Nats}$ be a family of nonnegative measurable functions on $\bspace$
and, for every ${i} \le {n} \in \Nats$, 
define 
\[
\indexmk {f(i)} {n} {m} {k} =
\begin{cases}
\indexmk f {n} {m} {k}, & {m}=i \\
0, & \text{otherwise,}
\end{cases}
&&
\indexmk {f[i]} {n} {m} {k} =
\begin{cases}
\indexmk f {n} {m} {k}, & {k}=i \\
0, & \text{otherwise.}
\end{cases}
\]
For every $i \le {n} \in \Nats$, 
define $\indicator i {n} \in \{0,1\}^{n}$ by $\indicator i {n } (j) = 1(i=j)$.
It is then straightforward to verify that 
\[
\lambda_{f} (s,\indicator i {n})
= \lambda_{f(i)} (s,\indicator i {n}),
\]
\[
\lambda_{f(i)} (s,\indicator j {n}) = 1,  \text { for $i \neq j$,}
\]
and thus
\[\label{thuisuhtiu}
1 - \Lambda_{f}(s,1) =
  \sum_{m \le {n}} (1 - \Lambda_{f(m)}(s,1)).
\]
By \cref{altlaplacefuncofjoint,thuisuhtiu}, 
\[
&- \frac 1 {n} \log \EE \exp \Oof {-\textstyle \sum_{{k} \le {m} \le {j} \le {n}} 
                                \indexmk {\ord X} {j} {m} {k} \indexmk f {j} {m} {k} }
\\&\qquad =  \NABM (1 - \Lambda_f (\argdot, 1))                                
\\&\qquad = \sum_{{m} \le {n}} \NABM (1 - \Lambda_{f(m)} (\argdot, 1)),
\]
which establishes the first claim.  To establish the second claim,
it is straightforward to verify that, for ${m} \le {n} \in \Nats$,
\[
1 - \lambda_{f} (s,\indicator {m} {n})
=   \sum_{{k} \le {m}} 
     \oF { 1 - \lambda_{f[k]} (s, \indicator {m} {n}) }
\]
and therefore     
\[
1 - \Lambda_{f}(s,1) 
&= \frac 1 {n}  
  \sum_{{m} \le {n}} 
    \oF { 1 - \lambda_{f} (s,\indicator {m} {n}) }
\\&= 
\frac 1 {n} 
  \sum_{{m} \le {n}} 
   \sum_{{k} \le {m}} 
     \oF { 1 - \lambda_{f[k]} (s, \indicator {m} {n}) }
\\&= 
\frac 1 {n} 
   \sum_{{k} \le {n}} 
  \sum_{{m} = {k}}^{n} 
     \oF { 1 - \lambda_{f[k]} (s, \indicator {m} {n}) }
\\&= 
\frac 1 {n} 
   \sum_{{k} \le {n}} 
  \sum_{{m} \le {n}} 
     \oF { 1 - \lambda_{f[k]} (s, \indicator {m} {n}) }
\\&= \sum_{{k} \le {n}} \of { 1 -  \Lambda_{f[{k}]}(s,1) } 
              . \label{esathuethues}
\]
The result follows as above from \cref{altlaplacefuncofjoint,esathuethues}.
\end{proof}

We next establish the law of the partial averages 
$\frac 1 {n} \indexm {\ord S} {n} {m} $
and
$\frac 1 {n} \indexk {\ord S} {n} {k} $:

\begin{thm}\label{thesumlawetuhui}
Let ${m},{k} \le {n} \in \Nats$.  Then
\[
\log \EE \exp \Oof{ - {\textstyle \frac 1 {n} } \indexm {\ord S} {n} {m} f }
= - \int \oooooF {
       \int (1-e^{-p\, f(s)})\, \Pr_{\eppf_s} \Ooooof{ \frac {N_{\nUni_{m},{n}}}{{n}} \in \dee p \wedge \nUni_{m} > \nUni_{{m}-1} }
     }  \, \NABM(\dee s)
\]
and
\[
\log \EE \exp \Oof{ - \frac 1 {n} \indexk {\ord S} {n} {k} f }
= - \int \oooooF {
       \int (1-e^{-p\, f(s)})\, \Pr_{\eppf_s} \Ooooof{ \frac {N_{{k},{n}}}{{n}} \in \dee p }
     }  \, \NABM(\dee s) 
\]
\end{thm}
\begin{proof}
The result follows in a straightforward fashion from \cref{altlaplacefuncofjoint}.
\end{proof}

We can now determine the distributional limits of the partial averages:

Note that
\[
  \PF {{m}-1} 0 (p) \, \sdist (s,\dee p) = \Pr_{\eppf_s} \Ooof{ P_{\nUni_{m}} \in \dee p \wedge \nUni_{m} > \nUni_{{m}-1} }.
\]

\begin{thm}
As $n \to \infty$,
the partial averages $\frac 1 {n} \indexm {\ord S} {n} {m}$ converge in distribution to a random measure $\ord \rH_{m}$ 
given by
\[\label{charofindexm}
\log \EE \exp \Oof{ - \ord \rH_{m} f }
&= - \int \oooooF {
       \int (1-e^{-p\, f(s)})\, \Pr_{\eppf_s} \Ooof{ P_{\nUni_{m}} \in \dee p \wedge \nUni_{m} > \nUni_{{m}-1} }
     }  \, \NABM(\dee s) .
\\
&
= - (\NABM \otimes \PF {{m}-1} 0 \sdist) (1-e^{-f'}), 
\]
where $f'(s,p) = p\, f(s)$.
\end{thm}
\begin{proof}
By the continuity of $\exp$ and \citep[][Thm.~16.16]{FMP2},
it suffices to prove that 
\[
\lim_{{n} \to \infty} \log \EE \exp \Oof{ - \frac 1 {n} \indexm {\ord S} {n} {m} f }  
= \log \EE \exp \Oof{ - \ord \rH_{m} f },
\]
for every nonnegative measurable $f$.
Define
\[
g_{n}(r) \defas \int (1-e^{-p\, r})\, \Pr_{\eppf_s} \Ooooof{ \frac {N_{\nUni_{m},{n}}}{{n}} \in \dee p \wedge \nUni_{m} > \nUni_{{m}-1} },
\quad \text{ for $r \in \NNReals$.}
\]
On $\{\nUni_{m} \ge \nUni_{{m}-1}\}$, we have $\frac {N_{\nUni_{m},{n}}}{{n}} \to P_{\nUni_{m}}$ \as\ and thus in distribution, and so,
by the boundedness and continuity of $p \mapsto (1 - e^{-p\, r})$ for $r \in \NNReals$,
we have
\[
\lim_{{n} \to \infty} g_{n}(r)
= g(r) \defas \int (1-e^{-p\, r})\, \Pr_{\eppf_s} \Ooof{ P_{\nUni_{m}} \in \dee p \wedge \nUni_{m} > \nUni_{{m}-1} }.
\]
Let $f$ be a nonnegative measurable function, %
and let $\bspace_1,\bspace_2,\dotsc$ be a partition of $\bspace$ such that $\NABM (\bspace_i) < \infty $ for every $i \in \Nats$.
Then, $g_n \circ f \le 1$, and so, by dominated convergence,
\[
\lim_{{n} \to \infty} \int_{\bspace_i} (g_n \circ f) \,  \dee\NABM 
=  \int_{\bspace_i} (g \circ f) \, \dee\NABM
\]
for every $i \in \Nats$, and so then also on all of $\bspace$, completing the proof.
\end{proof}

Define $\gsd k(s,\dee p) = \Pr_{\eppf_s} \Ooof{ P_{k} \in \dee p }$ for $k \in \Nats$.

\begin{thm}
As $n \to \infty$, the partial averages  $\frac 1 {n} \indexk {\ord S} {n} {k}$  converge in distribution to a random measure $\ord \rH^{k}$ given by 
\[\label{charofindexk}
\log \EE \exp \Oof{ - \ord \rH^{k} f }
&= ( \NABM \circ \gsd k ) (1 - e^{- f'}),
\\
&= - \int \oooooF {
       \int (1-e^{-p\, f(s)})\, \Pr_{\eppf_s} \Ooof{ P_{k} \in \dee p }
     }  \, \NABM(\dee s).
\]
where $f'(s,p) = p\, f(s)$.
\end{thm}
\begin{proof}
The proof is similar to that of the previous theorem.
\end{proof}

\newcommand{\SM}{\hat S}
\newcommand{\LSM}{\hat H}
We now will establish that the partial averages converge almost surely, and relate the limiting partial averages to the directing random measure $\ord \rH$.
For every $m \le n \in \Nats$, let $\indexm {\SM} {n} {m}$ be the random measure given by
\[
\indexm {\SM} {n} {m} (A) \defas \countproc {\indexm {\ord S} {n} {m}} (A \times (0,1]), \qquad A \in \bsa.
\]
It is straightforward to verify that \as, for $A \in \bsa$,
\[
\indexm {\SM} {n} {m} (A) = \# ( A \cap \supp \indexm {\ord S} {n} {m}  ).
\]
In the same manner, define $\indexk {\SM} {n} {k}$ for every $k \le n \in \Nats$. 
By construction we have that $\indexm {\SM} {n+1} {m} \ge \indexm {\SM} {n} {m}$ for all $n \in \Nats$ and so the limit
\[
\LSM_{m} \defas \lim_{n} \indexm {\SM} {n} {m}
\]
exists almost surely and is itself a random measure.  The same holds of
$\LSM^{k} \defas \lim_{n} \indexk {\SM} {n} {k}$.

\begin{thm}\label{limitingpoissonsupport}
Let $ {k},{m} \in \Nats$.  The limiting supporting measures
$\LSM_{m}$ and $\LSM^{k}$ are Poisson processes with intensities 
$ \oof{ \sdist \PF {{m}-1} {0} } \NABM$ 
and $ \gsd k \NABM$, respectively.
\end{thm}
\begin{proof}
Follows from the proofs of \cref{charofindexm,charofindexk}.
\end{proof}

\begin{remark}
Note that, in contrast, the support of $\ord \rH_{m}$ and $\ord \rH^{k}$, i.e., $\countproc{ \ord \rH_{m}}( \argdot \times (0,1])$ and  $\countproc {\ord \rH^{k}}(\argdot \times (0,1])$, are Poisson processes with intensities 
$ \oof{ \sdist \ooF{ 1_{(0,1]}\PF {{m}-1} {0} } } \NABM$ 
and $\oof { \gsd k 1_{(0,1]} } \NABM$, respectively.  In particular, the support of the limiting partial averages no longer contains points that were associated with \emph{dust} as they appear only once in the urn scheme and so their frequency converges to zero.
\end{remark}

\begin{cor}\label{disjointlimitingpoissonsupport}
The limiting supporting measures $\LSM_{m}$, for ${m} \in \Nats$, are mutually singular.  The same is true of $\LSM^{k}$, for ${k} \in \Nats$.
\end{cor}
\begin{proof}
Follows from independence of the measures (\cref{theindplawethiu}) and the fact that they are Poisson processes (\cref{thesumlawetuhui}).
\end{proof}

\begin{thm}
Let ${m},{k} \in \Nats$.
With probability one, for all $A \in \bsa$,
we have 
\[
\lim_{{n} \to \infty} \frac 1 {n} \indexm {\ord S} {n} {m} (A) = \ord H ( A \cap \supp \LSM_{m})
\quad \text{and} \quad
\lim_{{n} \to \infty} \frac 1 {n} \indexk {\ord S} {n} {k} (A) = \ord H ( A \cap \supp \LSM^{k}).
\]
\end{thm}
\begin{proof}
We prove the former case.  The latter follows identically.
Let $B$ be the support of $\LSM_{m}$, and let $A \in \bsa$.
Then, almost surely,
\[
\ord H (A \cap B)
&= \lim_{n \to \infty} \frac 1 n \ord S_{n} (A \cap B)
\\&= \lim_{n \to \infty} \frac 1 n \sum_{{m'}=1}^{n} \indexm {\ord S} {n} {m'} (A \cap B) \label{ethiutih}
\\&= \lim_{n \to \infty} \frac 1 n \indexm {\ord S} {n} {m} (A \cap B) \label{tihsuhtiuu}
\\&= \lim_{n \to \infty} \frac 1 n \indexm {\ord S} {n} {m} (A) \label{tihsuhtiuueue}
         ,
\]
where \cref{tihsuhtiuu} follows from \cref{disjointlimitingpoissonsupport}, and \cref{tihsuhtiuueue} follows from the fact that the fact that $\indexm {\SM} {n} {m} \,\upto\, \LSM_{m}$.
A Borel probability measure on $\bspace$ is characterized by its values on a countable collection $\cF \subseteq \bsa$ of measurable sets, and so the above holds \as\ simultaneously for any such collection $\cF$, and so it holds simultaneously for $\bsa$.
\end{proof}

Post facto, we may now take $\ord \rH_{m}$ and $\ord \rH^{k}$ to be not only the distributional limits of the partial averages but also their almost sure limits, i.e.,
\[
\ord \rH_{m} = \lim_{{n} \to \infty} \frac 1 {n} \indexm {\ord S} {n} {m}
\qquad \text{and} \qquad
\ord \rH^{k} = \lim_{{n} \to \infty} \frac 1 {n} \indexk {\ord S} {n} {k}
\]
almost surely, where the limits are understood in the strong sense.

\begin{remark}
The above development for $\ord \rH_{m}$ can be given a more direct proof.  Let ${m}\in \Nats$.  It is straightforward to show that, conditioned on 
$ \indexm {\ord X} {m} {m} $, the Bernoulli processes $ \indexm {\ord X} {m+1} {m}, \indexm {\ord X} {m+2} {m},\dotsc$ are conditionally i.i.d.  From this fact, the existence of the limiting partial average follows from the conditional version of the law of large numbers.  There is no obvious analogue of this approach for $\ord \rH^{k}$, hence the alternative development above.
\end{remark}

We now relate $\sum_{{m}=1}^\infty \ord \rH_{m}$ and $\sum_{{k}=1}^\infty \rH^{k}$ to $\ord \rH$.

\begin{thm}\label{thm:altreps}
With probability one,
\[\label{thedecomp}
\ord \rH 
= \Delta \NABM + \sum_{{m}=1}^\infty \ord \rH_{m} 
= \Delta \NABM + \sum_{{k}=1}^\infty \ord \rH^{k}.
\]
\end{thm}
\begin{proof}
Recall that $\ord \rH$ is the sum of a \nonrandomcomp\ component $\Delta \NABM$ and a completely random measure with \Levy\ measure $\NABM \otimes \PF {-1} {-1} \sdist$.  But, the infinite sum $\sum_{{m}=1}^\infty \ord \rH_{m}$ is also completely random with \Levy\ measure given by the infinite sum of the component \Levy\ measures, yielding 
\[\textstyle
\sum_{{m}=1}^\infty \ooof { \NABM \otimes \PF {{m}-1} 0 \sdist }
= \NABM \otimes (\sum_{{m}=1}^\infty \PF {{m}-1} 0) \sdist
= \NABM \otimes \PF {-1} {-1} \sdist,
\]
where the last equality follows from the identity 
\[\label{theidentityA}
p\inv = \sum_{{i}=0}^\infty (1-p)^{i}.
\]
Similarly, $\sum_{{k}=1}^\infty \ord \rH^{k}$ is completely random with \Levy\ measure
\[\textstyle
\sum_{{k}=1}^\infty \ooof { \NABM \otimes \gsd {k} }
= \NABM \otimes (\sum_{{k}=1}^\infty \gsd {k} )
= \NABM \otimes \PF {-1} {-1} \sdist,
\]
where the last equality follows from the identity
\[\label{theidentityB}
\EE f(P_1) = \EE \ooof { \textstyle \sum_{{k}=1}^\infty P_{k} f(P_{k}) }, \quad \text{ for measurable $f$ satisfying $f(0)=0$.}
\]
In particular, \cref{theidentityB}
implies that $\sdist(\dee p) = p\, (\sum_{k} \gsd {k})(\dee p)$ on $(0,1]$.
\end{proof}

\begin{remark}
\cref{simplestickbc,superposthm} follow immediately as a corollaries.   As noted in the Introduction following \cref{simplestickbc}, this result allows one to develop stick-breaking representations like those given by Paisley et al.~\citep{PZWGC2010} for the beta process.
\end{remark}

\begin{remark}\label{remarkaoeu}
Note that, on their own, the identities given in \cref{theidentityA,theidentityB} already yield, via a transfer argument and the calculus of completely random measures, the decomposition given by \cref{thedecomp}.
However, the relationship between $\ord \rH_{m}$ and $\ord \rH^{k}$, the limiting partial averages, and the underlying urn schemes is not revealed by this approach.  In particular, the random measures $\ord \rH^{k}$ are not measurable with respect to the process $\process X \Nats$. 
\end{remark}

\begin{remark}
In unpublished, independent work by James, Orbanz, and Teh~\citep{JOT}, 
an identity related to \cref{theidentityB} for the case $\Delta=0$ is shown to give rise to the 
decomposition of $\ord \rH$ in terms of $\ord \rH^{k}$ as in \cref{thedecomp} via a transfer argument (see \cref{remarkaoeu}). 
\end{remark}

\begin{remark}
To understand where the \nonrandomcomp\ component $\Delta\NABM$ arises, 
let $M \in \Nats$
and $A \in \bsa$. 
Then, with probability one,
\[
\ord H (A)
&= \lim_{n \to \infty} \frac 1 n S_{n} (A) 
\\&= \lim_{n \to \infty} \sum_{{m}=1}^{n} \frac 1 n \indexm S {n} {m} (A)
\\&= \sum_{{m}=1}^M \ord H_m (A)
         + \lim_{n \to \infty} \sum_{{m}=M}^{n} \frac 1 n \indexm S {n} {m} (A)
         .
\]
Taking $M \to \infty$ leaves the residual 
$\lim_{M \to \infty} \lim_{n \to \infty} \sum_{{m}=M}^{n} \frac 1 n \indexm S {n} {m} (A)$.
\end{remark}

We conclude by characterizing the approximation introduced by truncating the first infinite series in \cref{thm:altreps}.
Recall that
$\Delta_{m} = \sdist \PF {{m}-1} {0}$.

\begin{thm}
Assume that $\FixedAtoms = \emptyset$, let $m \in \Nats$, and
let
\[
\hat \rH  \defas \Delta \NABM + \sum_{m'=1}^{m-1} \ord \rH_{m'}
\]
be the finite truncation of $\rH$, i.e., the sum of only the first $m-1$ terms of the right hand side of \cref{thedecomp}.
Conditioned on $\rH$, let $X$ be a Bernoulli process with \hazard\ measure $\rH$ and let $\hat X$ be the restriction of $X$ to the complement of the support of $\rH - \hat \rH$.
Then the expected total mass of the ordinary component of $\rH-\hat \rH$, and equivalently, an upper bound on the probability that $\hat X \neq X$, is
\[\label{dabound}
\int_{\bspace} \oooooF{ \int_{(0,1]} (1-p)^{m-1} \sdist(s, \dee p) } \, \NABM(\dee s).
\]
When $\Delta = 0$, 
\cref{dabound} simplifies to
$
 \NABM \Delta_m.
$
\end{thm}
\begin{proof}
By Markov's inequality and then the chain rule of conditional expectation,
\[
\Pr \theset {  (X - \hat X)(\bspace) \ge 1 } 
          \le  \EE (X - \hat X)(\bspace) = \EE (\rH - \hat \rH)(\bspace).
\]
From \cref{thm:altreps,charofindexk}, 
\[
\EE ( \rH - \hat \rH)(\bspace) 
&= \sum_{m=k}^\infty \EE \ord \rH_{m} (\bspace) 
\\&= \sum_{{m}={k}}^\infty ( \NABM \otimes \PF {{m}-1} {1} \sdist) (\bspace \times (0,1] ))
\\&= ( \NABM \otimes ({\textstyle \sum_{{m}={k}}^\infty \PF {{m}-1} {1}}) \sdist) (\bspace \times (0,1] ))
\\&= ( \NABM \otimes \PF {{k}-1} {0} \sdist) (\bspace \times (0,1] )).
\]
The final claim follows from the definition of $\Delta_m$ and the fact that $\sdist (s,\{0\}) = 0$ if and only if $\Delta(s) = 0$.
\end{proof}
\begin{remark}[posterior approximation]
\cref{introapprox} follows immediately as a corollary.  
From \cref{theindplawethiu}, we can see that $\ord \rH_{n+1}, \ord \rH_{n+2}, \dotsc$ are independent of $\cF_n = \sigma(X_{[n]})$, for every $n \in \Nats$.  It follows that these approximation bounds also hold \emph{conditionally} on $\cF_n$ for $n < m$.
\end{remark}

\newcommand{\zno}{0^n\!1}
\section{Combinatorial structure and Indian buffet proccesses}
\label{sec:combstruct}

Here we study the combinatorial structure of a \COUS\ $\process X \Nats$ induced by a homogeneous family of EPPFs $\eppf_s = \eppf$ and an \iid\ sequence $\process Y \Nats$ of Poisson processes whose nonatomic mean measure has total mass $\gamma \in (0, \infty)$.  
This will lead us to generalizations of the Indian buffet process~\citep{GG05}.  Certain special cases recover processes proposed in the literature~\citep{GGS2007,TehGor2009a,BJP12}.

To begin, note that $\CombStruct{X_1}$ is entirely characterized by $M_1$, the cardinality of $X_1$.  
Because $X_1 = Y_1$ \as, the cardinality is Poisson-distributed with mean $\gamma$.  

In order to derive the distribution of $\CombStruct{X_{[n+1]}}$,
note that $\CombStruct{X_{[n]}}$ is $\CombStruct{X_{[n+1]}}$-measurable due to the fact that
\[
M_h = M_{h0} + M_{h1} \ \as, \qquad \text{for $h \in \Histories_n$.}
\] 
Moreover, by the complete randomness of $X_{n+1}$ given $X_{[n]}$, it follows that the random variables
$M_{h1}$, for $h \in \theset {0,1}^n$,
are conditionally independent given $\CombStruct {X_n}$.
(Recall that $M_{\zno}$ counts the number of atoms appearing for the first time in $X_{n+1}$, and 
$M_{h1}$, for $h \in \Histories_n$, counts the number of atoms appearing in $X_{n+1}$ with history $h$.)
Indeed, from \cref{maindefinetti},
we know that, conditioned on $\CombStruct {X_n}$,
\begin{enumerate}

\item
$M_{\zno}$ is Poisson-distributed with mean $\gamma \, \Delta_{n+1}$; and

\item
$M_{h1}$ is binomially-distributed with $M_h$ trials and mean $M_h \hat p$ 
where 
\[
\hat p = L_n(s,\, 0, {\textstyle \frac {s(h)} n}) 
= \Pr_{\eppf} \theset { h1\inv (1)  \in \Pi_n | h\inv (1) \in \Pi_n }
\]

\end{enumerate}

All together, 
we have
\[\label{ibpcond}
&\Pr \oF { \CombStruct{ X_{[n+1]}} = \of { m_h ; h \in \Histories_{n+1} }  
          \given \CombStruct{ X_{[n]} } = \of {m_h = m_{h1}+m_{h0} ; h \in \Histories_n } }
 \\&= \frac{ (\gamma\, \Delta_{n+1})^{m_{\zno}}} {(m_{\zno})!} e^{-\gamma\, \Delta_{n+1}}
 \\&\qquad  \times \prod_{h \in \Histories_n}  {m_h \choose m_{h1}} 
              L_n(s,\, 0, {\textstyle \frac {s(h)} n})^{m_{h1}}
              \,(1- L_n(s,\, 0, {\textstyle \frac {s(h)} n}))^{m_{h0}} 
\]

\begin{thm}[$\eppf$ Indian buffet process]\label{cupibp}
\[ \label{ibppmf}
\begin{split}
\Pr \theset { \CombStruct {X_{[n]}} = \of { m_h ; h \in \Histories_n } }  
&= 
\frac 
  {\gamma^{\sum_{h \in \Histories_n} m_h } \, 
  \exp \bigl [  - \gamma \sum_{j=1}^n \Pr_{\eppf} \theset {N_{1j} = 1} %
            \bigr ] } 
  {\prod_{h \in \Histories_n} (m_h)! } \\
&\quad\qquad \times                      
   \prod_{h \in \Histories_n}
     \ooof { \Pr_{\eppf} \theset { N_{1,s(h)} = s(h) \wedge N_{1n} = s(h) } }^{m_h}.
\end{split}
\]
\end{thm}
\begin{proof}
Note that  $N_{1,1}=1$ \as\ and so \cref{ibppmf}, for $n=1$, is precisely the statement that $M_1$ is Poisson distributed with mean $\gamma$, as was to be shown.
The result for $n > 1$ then follows by induction on $n$.  In particular, multiply \cref{ibppmf,ibpcond} and then apply the following identities:
\[
&\textstyle
m_{\zno} + \sum_{h \in \Histories_n} m_h 
= m_{\zno} + \sum_{h \in \Histories_n} (m_{h1} + m_{h0})
= \sum_{h \in \Histories_{n+1}} m_h,
\\
&
\textstyle
\frac 1 { \prod_{h \in \Histories_{n+1}} (m_h)! } 
=
   \frac 1 { \prod_{h \in \Histories_n} (m_h)! } 
  \frac 1 { (m_{\zno})! } 
 \prod_{h \in \Histories_n}  {m_h \choose m_{h1}} ,
\\
&
\textstyle
\Delta_{n} \defas \Pr_{\eppf} \theset {N_{1n} = 1}, \qquad \text{(by exchangeability),}
\]
and, by exchangeability,
\[
\begin{split}
&
\textstyle
\Pr_{\eppf} \theset { N_{1,n+1} = s(h) + z \given N_{1,s(h)} = s(h) \wedge N_{1,n} = s(h) }
\\
& \textstyle 
\qquad \times  \Pr_{\eppf} \theset { N_{1,s(h)} = s(h) \wedge N_{1,n} = s(h) } 
\\
& \textstyle 
=  \Pr_{\eppf} \theset { N_{1,s(hz)} = s(hz) \wedge N_{1,n+1} = s(hz) }
\end{split}
\]
for $z \in \theset {0,1}$.
\end{proof}

\begin{remark}
Note that 
\[
\Pr_{\eppf} \theset { N_{1,k} = k \wedge N_{1n} = k }
= \int_{[0,1]} p^{n-k} (1-p)^{k-1} \sdist(\dee p),
\]
and so \cref{introcombstruct} follows immediately as a corollary.
\end{remark}

Given its derivation from an exchangeable sequence,
there is, of course, no dependence on the ordering of the underlying sequence $\process X \Nats$ in the distribution of $\CombStruct { X_{[n]}}$ and indeed this is another way of noting that the combinatorial stochastic process is itself exchangeable in the following sense:

\begin{thm}[exchangeability]
Let $\sigma$ be a permutation of $[n]$, and, for $h \in \Histories_n$, consider the composition $h \circ \sigma \in \Histories_n$ given by 
$(h \circ \sigma)(j) = h(\sigma(j))$, for $ j \le n$.  Then $\nprocess M {\Histories_n} h \equaldist \gprocess {M_{h\circ \sigma}} {h \in \Histories_n}$.
\end{thm}

The following proof establishes the exchangeability directly.

\begin{proof}
Define $m_{n,k} \defas \sum_{h \in \Histories_n \st s(h) = k} m_h$ and note that $m_{n,k} = \sum_{h \in \Histories_n \st s(h \circ \sigma) = k} m_h $ because $s(h \circ \sigma) = s(h)$.  
The right hand side of \cref{ibppmf} can be written
\[ \label{ibppmfcount}
\frac 
{\gamma^{\sum_{k=1}^n m_{n,k} } \, 
  \exp \bigl [  - \gamma \sum_{j=1}^n \Pr_{\eppf} \theset {N_{1j} = 1} \bigr ] } 
{\prod_{h \in \Histories_n} (m_h)! }
 \prod_{k=1}^n
     \ooof { \Pr_{\eppf} \theset { N_{1k} = k \wedge N_{1n} = k } }^{m_{n,k} }.
\]
Finally, note that $h \mapsto h \circ \sigma$ is a permutation of $\Histories_n$, and so 
$\prod_{h \in \Histories_n} (m_h)! = \prod_{h \in \Histories_n} (m_{h \circ \sigma})!$.
\end{proof}

\subsubsection{Alternative representations of $\CombStruct { X_{[n]}}$ and their distributions}

A convenient way to represent $\CombStruct { X_{[n]}}$ is via a binary array/matrix $W$ such that, for every $h \in \Histories_n$, there are exactly $M_h$ columns of $W$ equal to $h$ (where $h$ is thought of as a column vector) and no all-zero columns.  The rows thus correspond to the measures $X_1,\dotsc,X_n$, and the columns to the pattern of sharing for some particular atom.  Note that it is possible that $W$ has zero columns, which corresponds to the case when $X_1=\dotsb=X_n=0$.

In order to determine a distribution over arrays, we must specify the ordering of the columns.
Griffiths and Ghahramani~\citep{GG05} developed the IBP using array representations and, doing so, introduced a canonical \defn{left-ordered form},
which can be defined as follows:  

\newcommand{\hle}{\prec}
\newcommand{\rhst}{H}
Write $\hle$ for the total order on $\Histories_n$ such that $g \hle h$ if and only if, for some $m \in [n]$ and all $i \le m$, we have $g(i)=h(i)$ and $g(m) > h(m)$.  
(That is, $\hle$ is the lexicographic ordering except that $1 \hle 0$.
As an illustration, $(1,1) \hle (1,0) \hle (0,1)$.)
The array $W$ is in left-order form if adjacent columns are either equal or ordered according to $\hle$.

More precisely, 
let $\rcard \defas \sum_{h \in \Histories_n} M_h$ denote the number of atoms among $X_1,\dotsc,X_n$,
and define 
\[
\rhst(j) \defas \sup \, \theset { h \in \Histories_n \st \textstyle \sum_{g \hle h} M_{g} < j}.
\]
(If $M_{(1,1)} = 0$, $M_{(1,0)}=3$, and $M_{(0,1)}=1$, then 
$\rcard =4$, 
$\rhst\inv (1,0) = \theset {1,2,3}$ and $\rhst\inv (0,1)=\theset {4,5,\dotsc}$ a.s.)
We may then express $W$ by
\[
W = [ \rhst(1) \dotsm \  \rhst(\rcard)  ],
\]
where each $\rhst(j)$ is viewed as a column vector.
That is $W \in \theset {0,1}^{n \times \rcard}$ \as, and 
$W_{ij} = 1$ implies that $X_i$ contains the atom labeled $j$, 
and those rows that also have a 1 in column $j$ correspond to the measures that share this atom.  

Because every feature allocation of $[n]$ corresponds with a unique binary array in left-ordered form, the probability of a realization of $W$ is precisely the probability of the (unique) realization of $\CombStruct { X_{[n]} }$ that gives rise to the left-ordered form.

Another ordering that has been studied is the \defn{uniform random labeling} \citep[][Pg.~9]{BJP13}. Informally, an array $W^*$ is labeled uniformly at random if it is equal to $W$ after a permutation of $W$'s columns which is uniformly distributed among all permutations of $[\rcard]$.
More carefully, let $U_1,U_2,\dotsc$ be an \iid\ sequence of uniformly-distributed random variables independent from $\CombStruct{X_{[n]}}$.  Associate with column $j$ of $W$ the label $U_j$ and then let $W^*$ be the array obtained by sorting the columns of $W$ in the \as\ unique increasing order of their labels.

Note that the number of distinct ways of ordering the $\rcard$ columns of $W$ is 
\[\label{combfactor}
\frac {(\rcard)!}{\prod_{h \in \Histories_n} (M_h)!}, 
\]
where the denominator arises from the fact that, for each column equal to $h$, there are $M_h$ indistinguishable copies.  This leads immediately to the following result:

\newcommand{\efpf}{\pi^*}
\begin{thm}\label{efpfthm}
Let $w \in \theset {0,1}^{n \times k} $ be a binary matrix with $k \ge 0$ nonzero columns and $n$ rows, and, for every $j \le k$, let $s_j \defas \sum_{i=1}^n w_{ij}$ be the sum of column $j$.
Then
\[\label{theefpf}
\Pr \theset {W^* = w }  
   &= \efpf(n; s_1,\dotsc,s_k) \\
   &\defas  \frac { \gamma^{k} }{k! } \, \exp \ooooF {  - \gamma \sum_{i=1}^n \Pr_{\eppf} \theset {N_{1j} = 1} }
         \prod_{j=1}^k  \Pr_{\eppf} \theset { N_{1s_j} = s_j  \wedge N_{1n} = s_j },
\]
where $\efpf(n;\argdot)$
is a symmetric function on $[n]^k$ for every $n \in \Nats$ and $k \in \NNInts$.
\end{thm}
\begin{proof}
The symmetry of $\efpf(n;\argdot)$ is manifest.  The result follows from dividing \cref{ibppmf} by \cref{combfactor} and then the definition of $s_j$.
\end{proof}

In the language of \citep{BJP13}, the functions $\efpf$ is an \defn{exchangeable feature probability function} or EFPF, which plays the role for exchangeable feature allocations analogous to that played by EPPFs for exchangeable partitions.
\cref{efpfthm} implies that every EPPF $\eppf$ induces an EFPF $\efpf$, via the distribution of $P_1$ induced by $\eppf$, which characterizes the combinatorial structure of a homogeneous continuum of urns with EPPF $\eppf$ and a nonatomic \hazard\ measure.

\section{Example: a continuum-of-two-parameter-urns scheme}
\label{sec:twoparam}

Teh and G\"or\"ur \citep{TehGor2009a} describe a three-parameter generalization of the IBP that exhibits power-law behavior by introducing a discount parameter that was understood to play a role similar to that of the discount parameter in 
the two-parameter Hoppe urn scheme \citep{MR515721} and its underlying combinatorial stochastic process, 
the two-parameter Chinese restaurant process (CRP)~\citep{MR1481784}.
The three-parameter generalization is shown to correspond to the combinatorial structure of an exchangeable sequence of Bernoulli processes directed by a class of random measures that Teh and G\"or\"ur called \defn{stable beta processes}.  Broderick, Jordan, and Pitman \citep{BJP12} study the same process and establish a number of asymptotic results characterizing the rate of growth of features, showing that they have power laws.

As we will see, the similarity between the combinatorial structure of the three-parameter IBP and the two-parameter CRP reflects a deeper connection: the three-parameter IBP
can be shown to correspond with the combinatorial structure of a \COUS\ $\process X \Nats$ induced by  the EPPF corresponding to the two-parameter CRP and a nonatomic mean measure.  By specializing the results of \cref{sec:cup}, we will make this connection precise.

\subsection{Two-parameter Chinese restaurant process}

\newcommand{\tpeppf}{\varpi} %
A well-studied EPPF is that corresponding with the two-parameter CRP.  In particular, consider the function $\tpeppf$ on $\Nats^*$ given by
\[
\tpeppf(n_1,\dotsc,n_k) = 
\frac {[\conc + \disc]_{k-1;\disc}}
        {[\conc+1]_{n-1}}
  \prod_{i=1}^k [1-\disc]_{n_i-1}
\]
where $\conc$ and $\disc$ satisfy
\[
\text{$0 \le \disc < 1$ and $\conc > -\disc$}
\]
or
\[
\text{$\disc = - k < 0$ and $\conc = m k$ for some $m = 2,3,\dotsc$ and $k > 0$;}
\] 
and 
\[
[x]_{m;a} \defas \begin{cases}
1 & \text{for $m=0$}, \\
x(x+a)\dotsm(x+(m-1)a)& \text{for $m\in \Nats$,}
\end{cases}
\]
and $[x]_m \defas [x]_{m;1}$.
The EPPF corresponding to the (one-parameter) CRP, which we introduced earlier in \cref{rem:CFopus}, is obtained by taking $\disc \downto 0$.

Let $\process Z \Nats$ be a $\tpeppf$-scheme with parameters $\conc$ and $\disc$ and assume $Z_1 \dist \Uniform$.  
$\process Z \Nats$ is also known as a two-parameter Hoppe urn.
The conditional distribution of $Z_{n+1}$ given $Z_{[n]}$ is
\[\label{tphu}
\frac {\conc + \nUni_n \cdot \disc} {\conc + n}  \, \Uniform
+
\sum_{j=1}^{\nUni_n} 
\frac {N_{jn} - \disc} {\conc + n} \, \delta_{\tilde Z_j} \,,
\]
where $\nUni_n$, $N_{jn}$, and ${\tilde Z_j}$ are defined as in \cref{sec:esp}.

\subsection{Directing random \hazard\ measure}

Let $\process X \Nats$ be a \COUS\ with \hazard\ measure $\BM$ and 
EPPF $\tpeppf$ with parameters $\conc$ and $\disc$.  
(One could also consider allowing $\conc$ and $\disc$ to vary across the space in a measurable way as in \cref{rem:CFopus}, but we will focus on the homogeneous case.)
It follows from~\cref{maindefinetti}
that there exists a random \hazard\ measure $\rH$ directing $\process X \Nats$, and that the law of $\rH$ is completely determined by $\tpeppf$ and $\BM$.  We now proceed to characterize $\rH$.

\subsubsection{Nonrandom component}

We have $\Pr_{\tpeppf} \theset { \sum_{i=1}^\infty P_i = 1 } = 1$, and so $\Delta = 0$. Therefore $\rH$ is \as\ purely atomic.

\subsubsection{Ordinary component}

We know that the distribution of the ordinary component of $\rH$ is determined by $\NABM$ and $\sdist$, i.e., the distribution of the \as\ limiting frequency $P_1$ of the first token.
Under the two-parameter Chinese restaurant process, it is known that the $P_1$ is beta-distributed with concentration $\conc+1$ and mean
$
\frac { 1 - \disc } {\conc + 1},
$
and thus $\sdist$ is absolutely continuous with respect to Lebesgue measure with probability density
\[
p \mapsto \frac{ \Gamma(\conc+1) } { \Gamma(1-\disc)\, \Gamma(\conc + \disc) } \, p^{-\disc} (1-p)^{\conc+\disc-1}.
\]
It follows that the intensity of the ordinary component of $\rH$ is 
\[\label{twoparamord}
(\dee s, \dee p) \mapsto \NABM(\dee s) \, \frac{ \Gamma(\conc+1) } { \Gamma(1-\disc)\, \Gamma(\conc + \disc) } \, p^{-1-\disc} (1-p)^{\conc+\disc-1} \dee p.
\]
Thus the ordinary component of $\rH$ is a so-called \defn{stable beta process}, as defined by \TG~\citep{TehGor2009a}.  
Note that when $\disc=0$, we recover the ordinary component of a beta process, as expected from \cref{rem:CFopus}. 
\TG\ show that a stable beta process underlies a three-parameter IBP scheme, which we will derive from the \COUS\ perspective below.
The connection with the two-parameter CRP is now manifest.

\subsubsection{The \fixedcomp\ component}

If $\BM$ is nonatomic, then the directing \hazard\ measure $\rH$ is composed of only an ordinary component.  
The continuum of urns perspective, however, also characterizes $\rH$ when $\BM$ has atoms, which will be seen to be useful when we define a \defn{hierarchical stable beta process}.  
Recall that, for an atom $s \in \FixedAtoms$ of $\BM$ such that $\BM \theset s = q$, we have
\[\label{kerneldefsbp}
\rH\theset s \dist \bkernel_{0,0}(q, \argdot) = \Pr_{\tpeppf} \theset { Q_q \in \argdot }.
\]
Recall that $Q_q = \drm [0,q]$.
When $\disc=0$, we know that $\drm$ is a Dirichlet process and so 
\[
(\drm[0,q],\, \drm(q,1])
\]
is Dirichlet-distributed with concentration $\conc$ and mean vector $(q,1-q)$, and thus
\[
\rH \theset s \dist \Beta(\conc\, \BM \theset s,\conc\,(1- \BM \theset s)),
\]
which agrees with the definition of the beta process~\citep{MR1062708,Thibaux2007}.

For a general two-parameter CRP, we cannot say as much.  When $\disc = -k < 0$ and $\conc>k$ is an integer multiple $m$ of $k$, it is understood that $\drm$ will be supported on a finite \iid\ set $\theset{\tilde Z_1,\dotsc,\tilde Z_m}$ with the probability mass symmetrically Dirichlet-distributed \NA{with concentration $k\inv$.}
Letting $M=\sum_{j=1}^m \delta_{\tilde Z_j} [0,q]$, we have
\[
\drm [0,q] \given M \dist \Beta( M k, (m - M) k),
\]
where $M$ is binomially-distributed with $m$ trials and mean $mq$. 
Unlike the case when $\disc\ge 0$, there is positive probability that $\rH\theset s \in \theset {0,1}$, 
and so the distribution of $\rH\theset s$ is not even absolutely continuous, 
although it is a mixture of beta distributions, and hence absolutely continuous on the event $\theset { 0 < \rH\theset s < 1 }$.

When $\disc>0$ and $\conc > -\disc$, the distribution of $\rH\theset s$ can be simulated exactly to arbitrary accuracy using the stick-breaking characterization of $\process P \Nats$.  In particular, there is a collection of independent random variables $\process W \Nats$ with $W_j \dist \Beta(1-\disc,\conc+j \,\disc)$ such that
\[
P_j = \left [ \, \prod_{i=1}^{j-1} (1-W_i) \right ] W_j \ \ \as
\]
By the definition of $Q_q$, we know that there exists an \iid\ process $\process T \Nats$ of Bernoulli random variables with mean $q$, independent of $\process W \Nats$, such that $Q_q = \sum_{i=1}^\infty P_i T_i$ a.s.
For any $\varepsilon> 0$, we can truncate this sum at a finite level and compute an approximation $\hat Q_q$ such that $|\hat Q_q - Q_q | < \varepsilon$ a.s.  By including additional terms, which we can compute as needed, this approximation can be tightened on demand.
(The framework of computable probability theory would allow us to make more precise statements about computability.  In particular, $Q_q$ has a computable distribution, uniformly in $q \in [0,1]$, in the sense of~\citep{MR1694441}.  See~\citep{Roy2011} for more details.)

Despite the explicit sampling rule for $\rH\theset s$, there appears to be no simple expression for its distribution in terms of $\BM \theset s$, $\conc$, and $\disc$, although the work of James, Lijoi, and Pr\"unster~\citep{MR2398765} on the distribution of linear functionals of $\drm$ provides an approach to this problem.  For example, using \citep[][Thm.~2.1]{MR2398765}, one can identify situations where the distribution of $\rH \theset s$ is absolutely continuous with respect to Lebesgue measure.

\subsection{Conditional law}

\cref{maindefinetti} characterizes the conditional law of $\rH$ given $X_{[n]}$ in terms of the kernel $\bkernel$.  The distributions $\bkernel_{n,k}(p,\argdot)$, for $p > 0$,
are simulable along the lines described above, and these simulations can be used to produce MCMC algorithms for more complicated computations.  On the other hand, it is straightforward to show that
\[\label{condlawoftwop}
\bkernel_{n,k}(0, \argdot) = \Beta( k - \disc,\, n - k + \conc + \disc),
\]
which is the conditional distribution of the mass of an atom of $\rH$ appearing $k$ times among the ordinary components of $X_{[n]}$.

\subsection{Connection with the three-parameter IBP}

Assume that $\BM$ is nonatomic, and let $\gamma \defas \BM(\bspace) > 0$ denote the total mass of the \hazard\ measure.
Because $\BM$ is nonatomic, $\rH$ has only an ordinary component, which we know to be that of a stable beta process, and so, we may conclude from the work of Teh and G\"or\"ur that the combinatorial structure of $\process X \Nats$ must be that of the three-parameter IBP.  

Regardless, it is instructive to revisit the probability mass function of the combinatorial structure given by \cref{cupibp} in this special case, as we see
that it depends on the EPPF only through the probabilities
$\Pr_{\tpeppf} \theset { N_{1k} = k \wedge N_{1n} = k }$, for $1 \le k \le n \in \Nats$.

By exchangeability, $\Pr_{\tpeppf} \theset { N_{1n} = 1} = \Pr_{\tpeppf} \theset { \nUni_n > \nUni_{n-1} }$, and so both are equal to the probability of the event that a new token appears on the $(n+1)$-st iteration of the two-parameter Hoppe urn (equivalently, a new table being allocated in a two-parameter Chinese restaurant process). 
Recall that
$
\Delta_n 
\defas \Pr_{\tpeppf} \theset{  \nUni_{n+1} > \nUni_n }.
$
We have 
\[
\Delta_n =  \EE_{\tpeppf} \of{ 
\frac {\tpeppf(\nmul_{n}^{+(\nUni_n+1)})}
        {\tpeppf(\nmul_n)}
   }
 = \EE_{\tpeppf} \of { \frac {\conc + \nUni_{n}\cdot \disc} {\conc + n} }.
\]
Conditioning on $K_n$ and then averaging, we have
\[
\begin{split}
\Delta_{n+1} 
&= \EE_{\tpeppf} \Bigl[
\frac {\conc + \nUni_{n}\cdot \disc} {\conc + n}  \cdot  
\frac {\conc + (\nUni_{n}+1)\cdot \disc} {\conc + n + 1}        
\\
&\qquad\qquad+ \left(1-\frac {\conc + \nUni_{n}\cdot \disc} {\conc + n} \right)\cdot
\frac {\conc + \nUni_{n}\cdot \disc} {\conc + n + 1}     
\Bigr]  %
= \Delta_{n} \cdot \frac {\conc + \disc + n} {\conc + 1 + n}, 
\end{split}
\]
and thus
\[\label{newcust}
\Delta_{n} 
= \frac{[\conc+\disc]_{n}}{[\conc+1]_{n}} 
= \frac{\Gamma(\conc+1)\,\Gamma(n+\conc+\disc)}{\Gamma(n+\conc+1)\,\Gamma(\conc+\disc)} \ .
\]
It follows that the number of new features appearing at stage $n+1$ is $\gamma \, \Delta_n$.
Teh and G\"or\"ur \citep{TehGor2009a} derived the distribution on the number of new features (and in particular \cref{newcust}) from \cref{twoparamord} via the calculus of completely random measures, but the connection with the combinatorics of an underlying two-parameter model was not made.

By exchangeability,
the probability 
$\Pr_{\tpeppf} \theset { N_{1k} = k \wedge N_{1n} = k }$
is also that of the event that a two-parameter Hoppe urn admits a new token on the $(n-k+1)$-st iteration, and then admits $k-1$ additional copies of this token in a row.
Admitting a new token occurs with probability $C_{n-k}$, and admitting this new token $k-1$ additional times occurs with probability
\[\label{eq:staying}
\frac {1-\disc} {\conc + n-k+1} \dotsm \frac {k-1-\disc} {\conc + n-1} 
= \frac { \Gamma(k-\disc)} { \Gamma(1-\disc)}
   \frac { \Gamma(\conc+n-k+1) } { \Gamma(\conc+n) } ,
\]
and so 
\[
\Pr_{\tpeppf} \theset { N_{1k} = k \wedge N_{1n} = k }
&=
\frac { \Gamma(k-\disc)} { \Gamma(1-\disc)}
\frac { \Gamma(\conc+1) } { \Gamma(\conc+n) } 
\frac {\Gamma(n-k+\conc+\disc)}{\Gamma(\conc+\disc)} \ ,  
\]
which, as expected, leads to a probability mass function that agrees with~\citep[][Eq.~(10)]{TehGor2009a}. 
The terms appearing in \cref{eq:staying} are connected to the \COUS\ by noting that, informally speaking,
conditioned on $X_1 \theset s = 1$, the probability that $X_{n+1} \theset s = 1$ is
\[\label{oldcust}
\frac { N_{1n} - \disc } {\conc + n}
= \frac { \Bigl ( \sum_{j \le n} X_j \theset{s} \Bigr) - \disc} {\conc + n} 
\]
By exchangeability, the right hand side of \cref{oldcust} is also the probability conditioned on $X_j \theset s = 1$ for some $j \in [n]$, and so governs the probability of an element recurring given that it has already appeared.

To summarize, we have:
\begin{thm}[combinatorial structure of two-parameter urn scheme]\label{ibptwo}
Let $\process X \Nats$ be a \COUS\ with nonatomic \hazard\ measure $\NABM$ and EPPF $\tpeppf$.  
Then $\CombStruct { \process X \Nats }$ is a three-parameter IBP with mass parameter $\gamma \defas \NABM(\bspace)$, concentration parameter $\conc$ and discount parameter $\disc$.
\end{thm}

\subsection{Hierarchical stable beta processes}
\label{ssec:hsbp}

It is worth mentioning that Teh and G\"or\"ur do not explicitly
propose a definition 
for the \fixedcomp\ component of a stable beta process, 
although they did establish the beta law \cref{condlawoftwop}.
From a conjugacy perspective, 
it might then seem natural to consider a \fixedcomp\ component governed by a kernel
of the form
\[\label{betakernel}
q \mapsto  \Beta( f(q,\conc,\disc), g(q,\conc,\disc) ),
\]
for suitable functions $f,g > 0$.
In an article on beta negative binomial processes, 
Broderick, Mackey, Paisley, and Jordan~\citep{BJP11} make such a proposal, taking
\[
f(q,\conc,\disc) \defas \conc \, \gamma \, q - \disc > 0 \quad\text{and}\quad g(q,\conc,\disc) \defas \conc \, (1-\gamma\, q) + \disc > 0
\]
for some $\gamma > 0$. 
Except in a trivial case, no such kernel---even one of the general form given in \cref{betakernel}---corresponds with a \COUS\ with a stable beta process ordinary component.

\begin{thm}
Consider a \COUS\ with a kernel of the form given in \cref{betakernel} for some functions $f,g > 0$.
Then the ordinary component is given by \cref{twoparamord} 
only if $\disc = 0$, i.e., 
only if $\rH$ is a beta process.
\end{thm}
\begin{proof}
Let $h(q,\conc,\disc) \defas g(q,\conc,\disc) / f(q,\conc,\disc)$ and let $\kernel'$ denote the kernel given by \cref{betakernel}.
We have $\lim_{q \downto 0} \kernel'(q) = \delta_0$, which implies that $\lim_{q \downto 0} h(q,\conc,\disc) = \infty$.  On the other hand, we have that the law $p\, q\inv\, \kernel'(q,\dee p)$ converges weakly to $\beta(1-\disc,\conc+\disc)$ as $q\downto 0$.  This implies that $f(q,\conc,\disc) \to -\disc$ and $g(q,\conc,\disc) \to \conc + \disc - 1$.  Together, these imply that $\disc = 0$, which is then seen to be sufficient by noting that this is simply a beta process which has the form \cref{betakernel} for $f(q,\conc) = q \conc$ and $g(q,\conc) = \conc (1-q)$.
\end{proof}

\subsection{Other generalizations}

In this section we have connected 
existing work on exchangeable sequences of Bernoulli processes directed by stable beta processes
with continuum of two-parameter Hoppe urn schemes.
Another category of EPPFs that would be natural to investigate are those of so-called \defn{Gibbs-type}~\citep{MR2160320}.

\section{The continuum limit}
\label{sec:climit}

The following theorem states that the limiting distribution of the discrete models presented in the introduction is indeed that of a \COUS.

\begin{thm} 
Assume $\eppf$ is continuous on $\bspace$.
Let $\BM^1,\BM^2,\dotsc$ be a sequence of purely-atomic \hazard\ measures strongly converging to $\BM$, i.e., for every $A \in \bsa$, 
\[
\BM^m (A) \to \BM(A) \quad\text{as}\quad m \to \infty.
\]
For every $m$, let $\process {X^m} \Nats$ be a \COUS\ with \hazard\ measure $\BM^m$, and  
let $\process X \Nats$ be a \COUS\ with \hazard\ measure $\BM$.
Then $\process {X^m} \Nats$ converges in distribution to $\process X \Nats$ as $m \to \infty$.
\end{thm}
\begin{proof}
By \citep[][Thm.~4.29]{FMP2}, it suffices to show that $X^m_{[n]}$ converges to $X_{[n]}$ in distribution as $m \to \infty$ for every $n \in \Nats$.   Fix $n \in \Nats$.
It is straightforward to show that
$X^m_{[n]}\{s\}$ converges in distribution to $X_{[n]}\{s\}$ as $m \to \infty$ for every $s \in \FixedAtoms$.
As the measures are completely random, it therefore suffices to prove convergence  on the complement $\bspace \setminus \FixedAtoms$, and so we will assume without any loss of generality that $\FixedAtoms = \emptyset$.
Moreover, $\BM$ is $\sigma$-finite and so we can partition $\bspace$ into a countable partition $\bspace_1,\bspace_2,\dotsc$ such that $\BM(\bspace_k)$ is finite for every $k \in \Nats$.  The restrictions of $X_{[n]}, X^1_{[n]},X^2_{[n]},\dotsc$ to each subset $\bspace_k$ are independent, and so we can, without loss of generality, assume also that $\BM$ is finite.

For every $m \in \Nats$, let $\process {Y^m} \Nats$ be the \iid\ sequence of Poisson processes underlying $\process {X^m} \Nats$.  Fix $n \in \Nats$ and define $R^m_n = Y^m_1 + \dotsb + Y^m_n$.  Define $\process Y \Nats$ and $R_n$ similarly.
From \cref{altlaplacefuncs}, we know there is a probability kernel $\nu$ satisfying $\nu(R_n) = \Pr [ X_{[n]} | R_n]$ \as\ and $\nu(R^m_n) = \Pr[X^m_{[n]}| R^m_n]$  \as\ for every $m \in \Nats$.
The claim is that $\nu$ is a continuous map from the subspace $\mathcal N$ of locally-finite simple Borel measures
to the space of Borel measures on $\mathcal N^n$.
Then so is the map taking the distribution of a random element $R$ in $\mathcal N$ to the distribution of $\nu(R)$.
Therefore by \citep[][Thm.~16.24]{FMP2}, 
it suffices to show that $R^m_n$ converges to $R_n$ in distribution.  

By independence of the $Y_n$, the random measure $R_n$ is a Poisson process with intensity $n \NABM$, and so, by~\citep[][Thm.~16.17]{FMP2}, it suffices to show that
1) $\Pr \theset{ R^m_n A = 0 } \to \Pr \theset{ R_n A = 0 }$ for every $A \in \bsa$ and 2) $\limsup_n \EE(R^m_n K) \le \EE(R_n K) < \infty$ for all compact sets $K \subseteq \bspace$.  

We have 
\[
\Pr \theset{  R^m_n A = 0 }
&= \prod_{s \in \Atoms_m \cap A} \Pr \theset{ R^m_n \theset s = 0 }
= \prod_{s \in \Atoms_m \cap A} e^{- n \BM^m \theset s}
\\&= e^{- n \sum_{s \in \Atoms_m \cap A} \BM^m \theset s}
= e^{- n \BM^m (A)}
\to e^{- n \BM (A)} = \Pr \theset{ R_n A = 0 },
\]
which establishes (1).  For (2), note that
\[
\EE(R^m_n K) 
&= \sum_{s \in \Atoms_m \cap K}  \!\! \EE R^m_n \theset s 
= \sum_{s \in \Atoms_m \cap K} \!\! \BM^m \theset s  
= \BM^m K \to \BM K = \EE (R_n K) < \infty,
\]
completing the proof.
\end{proof}

\begin{remark}
The proof fails if the convergence is merely weak.  To see this, take $\BM^m \defas \frac 1 2 ( \delta_{m^{-1}} + \delta_{m^{-2}})$.  Then $\BM^m \to \delta_0$ weakly, but $R^m_1$ converges weakly to a point process concentrated on $\{0\}$ whose total mass is binomally distributed with mean $1$ and variance $\frac 1 2$.  In contrast, a Bernoulli process with mean $\delta_0$ is almost surely $\delta_0$ itself.
\end{remark}

Having established the relationship between the continuum limit described in \cref{sec:intro} and the \COUS, \cref{intromain,genbeta,superposthm,introapprox,introcombstruct} now follow as special cases from their counterparts in \cref{sec:cup}.

\scriptsize

\section*{Acknowledgments}

The author would like to thank 
Nate Ackerman,
Cameron Freer, 
Zoubin Ghahramani, 
and
Creighton Heaukulani 
for feedback on drafts.
This research was carried out in part while the author was a research fellow of Emmanuel College, Cambridge, with funding from a Newton International Fellowship through the Royal Society.

\printbibliography

\normalsize

\appendix

\section{Transfer arguments}
\label{sec:transfer}

Transfer arguments translate distributional equalities into existence claims for random variables on extensions of the underlying probability space.  The interested reader is advised to consult \citep[][Chp.~6]{FMP2}.

\begin{thm}[transfer {\citep[][Prop.~6.10]{FMP2}}]\label{transfer}
For any measurable space $S$ and Borel space $T$, let $\xi \equaldist \tilde \xi$ and $\eta$ be random elements in $S$ and $T$, respectively. Then there exists a random element $\tilde \eta$ in $T$ with $(\tilde \xi,\tilde\eta) \equaldist (\xi,\eta)$.  More precisely, there exists a measurable function $f : S \times [0,1] \to T$ such that we may take $\eta = f(\tilde \xi, \vartheta)$ whenever $\vartheta \ind \tilde\xi$ is $\Uniform$.
\end{thm}

\begin{cor}[stochastic equations {\citep[][Prop.~6.11]{FMP2}}]\label{stoceqn}
Fix two Borel spaces $S$ and $T$, a measurable mapping $f : T \to S$ , and some random elements $\xi$ in $S$ and $\eta$ in $T$ with $\xi \equaldist f(\eta)$. Then there exists a random element $\tilde \eta \equaldist \eta$ in $T$ with $\xi = f(\tilde \eta)$ a.s.
\end{cor}

\begin{lem}[conditional independence and randomization {\citep[][Prop.~6.13]{FMP2}}]\label{rand}
Let $\xi$, $\eta$, and $\zeta$ be random elements in some measurable spaces $S$, $T$, and $U$, respectively, where $S$ is Borel.  Then $\xi \ind_\eta \zeta$ iff $\xi = f(\eta,\vartheta)$ a.s.\ for some measurable functions $f : T \times [0,1] \to S$ and some $\Uniform$ random variable $\vartheta \ind (\eta, \zeta)$.
\end{lem}

\ \vfill

\end{document}